\documentclass[reqno]{amsart}

\usepackage{amsfonts}
\usepackage{amssymb}
\usepackage{amsmath}
\usepackage{enumitem}
\usepackage{xcolor}
\usepackage{todonotes}

\setenumerate{label={\rm (\alph{*})}}
\usepackage{bbm,euscript,mathrsfs}
\usepackage{paper_diening}
\usepackage{graphicx}
\usepackage[top=1in, bottom=1.25in, left=1.10in, right=1.10in]{geometry}

\allowdisplaybreaks

\numberwithin{equation}{section}

\usepackage{tikz-cd}
\usetikzlibrary{lindenmayersystems}
\usetikzlibrary{decorations.pathreplacing}

\DeclareMathOperator{\BMO}{BMO}

\DeclareMathOperator{\Div}{div}

\newcommand{\R}{\mathbb R}
\newcommand{\N}{\mathbb N}
\newcommand{\dd}{\mathrm d}

\newcommand{\dx}{\,\mathrm{d}x}

\newtheorem{theorem}{Theorem}[section]

\newtheorem{corollary}[theorem]{Corollary}
\newtheorem{remark}[theorem]{Remark}

\theoremstyle{definition}
\newtheorem{definition}[theorem]{Definition}

\begin{document}

\title[Stokes equations in rough domains]
{A Schauder theory for the Stokes equations in rough domains}

\author{Dominic Breit}
\address{Institute of Mathematics, TU Clausthal, Erzstra\ss e 1, 38678 Clausthal-Zellerfeld, Germany}
\email{dominic.breit@tu-clausthal.de}
%\thanks{The authors would like to thank S. Schwarzacher and E. S\"uli for valuable suggestions.}

%\author{Prince Romeo Mensah}
%\address{Gran Sasso Science Institute, Viale F. Crispi, 7. 67100 L'Aquila, Italy}
%\email{p.mensah@imperial.ac.uk; romeo.mensah@gssi.it}

%    General info

\date{\today}



\begin{abstract}
We consider the steady Stokes equations in a bounded domain with forcing in divergence form supplemented with no-slip boundary conditions.
We provide a maximal regularity theory in Campanato spaces (inlcuding $\mathrm{BMO}$ and $C^{0,\alpha}$ for $0<\alpha <1$ as special cases) under minimal assumptions on the regularity of the underlying domain. Our approach is based on pointwise multipliers in Campanto spaces.
\end{abstract}

\maketitle

\section{Introduction}

For a given forcing $\bfF:\Omega\rightarrow\R^{n\times n}$
we consider the Stokes system
 \begin{align}\Delta\bfu-\nabla \pi=-\Div\bfF,\quad \Div\bfu=0,&\label{3}
 \end{align}
in a bounded domain $\Omega\subset\R^n$ supplemented with no-slip boundary conditions. We aim at a maximal regularity theory which allows estimates of the form
\begin{align}\label{eq:aim}
\|\nabla\bfu\|_{X(\Omega)}+\|\pi\|_{X(\Omega)}\lesssim \|\bfF\|_{X(\Omega)}
\end{align}
in a scale of function spaces $X(\Omega)$ ranging from $\BMO(\Omega)$ (functions with bounded mean oscillation) to $C^{0,\alpha}(\Omega)$ ($\alpha$-H\"older-continuous functions) under minimal assumptions on the regularity of $\partial\Omega$. We will introduce this scale by means of Campanato spaces in detail in Section \ref{subsec:cs}. They are based on oscillatory integrals with weights, where the weight for $\BMO$ is $\omega(r)=1$ and the weight for  $C^{0,\alpha}$ is $\omega(r)=r^\alpha$.

The regularity theory for \eqref{3} has a long history and one typically considers rather estimates of the form
\begin{align}\label{eq:aim'}
\|\nabla^2\bfu\|_{X(\Omega)}+\|\nabla\pi\|_{X(\Omega)}\lesssim \|\bff\|_{X(\Omega)},\quad \bff=\Div\bfF.
\end{align}
If $\partial\Omega$ is sufficiently smooth and $X(\Omega)$ is a non-borderline function space estimates \eqref{eq:aim} and \eqref{eq:aim'} are basically equivalent. This can be shown be means of the square root of the Stokes operator. A classical problem concerns \eqref{eq:aim'} in the class $X(\Omega)=L^q(\Omega)$ (or $W^{k,q}(\Omega)$ for $k\in\N$) for $1<q<\infty$. For an exhaustive picture and detailed references we refer to \cite[Chapter IV]{Ga}. The classical assumption here is that $\partial\Omega$ has a $C^2$-boundary (a $C^{2+k}$-boundary, respectively). Under minimal assumptions concerning the boundary regularity a corresponding theory has been developed only very recently in \cite{Br}. This is based on the theory of Sobolev multipliers \cite{MaSh} which has been employed before successfully to obtain regularity estimates for the Laplace equation in non-smooth domains. Similarly, estimate \eqref{eq:aim'} is well-known in the case $X=C^{0,\alpha}$ for $0<\alpha<1$ for smooth domains, cf.~\cite{So}. Related results under minimal assumptions on the regularity of the boundary seem missing in literature. As far as the space $\BMO$ is concerned, results on the semigroup associated to the unsteady Stokes equations can be found in \cite{BG}, \cite{BGS} and \cite{BGST}. Form that one can easily derive estimates in the spirit of \eqref{eq:aim'}. The results from \cite{BG}, \cite{BGS} and \cite{BGST} require, however, rather smooth domains.
% The conclude this historical overview let us mention the results from \cite{BS} (building up on \cite{FKV}), where the particular case of Lipschitz domains is considered. The result is an estimate in fractional Sobolev space and thus lies in between \eqref{eq:aim} and \eqref{eq:aim'}. 

A comprehensive theory for the $p$-Laplace system (which includes the plain Laplacian for $p=2$) that relates to \eqref{eq:aim} has been developed recently in \cite{BCDS}. The sharpness of the results (in the sense of minimal regularity of the underlying domain) is demonstrated by a family of apropos examples. As one might have expected the required regularity for an estimate in $C^{0,\alpha}$ is a boundary of class $C^{1,\alpha}$, meaning that $\partial\Omega$ can be locally described by charts of class $C^{1,\alpha}$ (we will make this concept precise in Definition \ref{def:besovboundary}). The question about the correct assumption for the $\BMO$-estimate is more subtle and it is unclear what to expect. It turned out that the correct assumption is that the derivatives of the boundary charts belong to a Campanato space with a weight which decays faster
than $\sigma(r)=-1/\log(r)$. As a matter of fact, for each weight $\omega$, which generates the Campanato space $\mathcal L^\omega$, one can compute a weight $\sigma$ via the formula
\begin{align}\label{eq:sigmaA}
\sigma(r):=\omega(r)\bigg(\int_{r}^1\frac{\omega(\rho)}{\rho}\,\dd\rho\bigg)^{-1}
\end{align}
 generating the optimal Campanato space $\mathcal L^\sigma$ for the boundary charts. For $\BMO$ (where $\omega(r)=1$) or function spaces close to it the space $\mathcal L^\sigma$ is significantly smaller than $\mathcal L^\omega$, while both coincide in the case $\omega(r)=r^\alpha$.

The aim of the present paper is now to prove estimate \eqref{eq:aim} under the same assumption on $\partial\Omega$ as for the ($p$-) Laplacian in \cite{BCDS}. Unfortunately, the approach from \cite{BCDS} does not extend to the Stokes system. It is based on a flattening of the boundary and a reflection of the hence obtained solution at the flat boundary. In the case of no-slip boundary conditions one must apply an odd reflection which destroys the divergence-free constraint.\footnote{The situation changes when perfect-slip boundary conditions are considered and a corresponding theory in this case has been obtained in \cite{MS} and applies even to nonlinear Stokes systems.}  Thus we apply a different strategy which is inspired by the recent results from \cite{Br}, where an approach by Sobolev multipliers lead to optimal results for estimates in (fractional) Sobolev spaces: We employ pointwise Campanato multipliers, see Section \ref{subsec:cs} for details.
They have been fully characterised in \cite{Na} and lead to exactly the same relationship between the spaces $\mathcal L^\omega$ and $\mathcal L^\sigma$ as in  \cite{BCDS}: Estimate \eqref{eq:aim} holds if the derivatives of the boundary charts are multipliers on $\mathcal L^\omega$ and thus must belong to $\mathcal L^\sigma$ with $\sigma$ given in \eqref{eq:sigmaA}. This is our main result and can be found in Theorem \ref{thm:stokessteady}.

 Before the space of Campanato multipliers comes into play, one has to introduce local coordinates and apply an estimate for the Stokes problem on the half space. We analyse the latter in Section \ref{sec:half} and reduce the question of Campanto-regularity to decay estimates (such a strategy has also been applied in \cite{DKS,DKS2} and, eventually, in \cite{BCDKS,BCDS}). While they are well-known in the interior, we were unable to trace a reference for decay estimates at boundary points and thus give a proof in Theorem \ref{thm:decay}. The situation is more delicate than in the case of classical elliptic problems studied in \cite{Ca}: The behaviour of velocity gradient and pressure function cannot be separated
and the osciallations of the former can only be controlled with the help of the oscillations of the latter. 
This entirely differs from the case of interior estimates, where the pressure function can be completely taken out of the picture (see, e.g., \cite{GM} or \cite[Lemma 3.5]{FuS}).

By a strategy similar to that just described for estimate \eqref{eq:aim} we are also able to prove estimate \eqref{eq:aim'} in Campanto spaces $\mathcal L^\omega$ for domains of minimal regularity requiring. Here we basically require one derivative more for the boundary charts, cf. Theorem \ref{thm:stokessteadyf}. It is interesting to note that on this level the gap between the function space $\mathcal L^ \omega$ in the estimate and that for the boundary charts (that is $\mathcal L^\sigma$ in the case of \eqref{eq:aim} with $\mathcal L^\sigma$ possibly smaller than $\mathcal L^\omega$) does not even exists in the $\BMO$-case. If the second derivatives of the boundary charts belong to $\mathcal L^\omega$, i.e., $\partial\Omega\in W^2\mathcal L^\omega$, then estimate \eqref{eq:aim'} holds. In particular, $\partial\Omega\in W^2\BMO$ implies \eqref{eq:aim'} for $\BMO$.
This significantly improves
the results which follow from \cite{BGST}, where a $C^3$-boundary is assumed.

\section{Preliminaries}
\subsection{Conventions}

We write $f\lesssim g$ for two non-negative quantities $f$ and $g$ if there is a $c>0$ such that $f\leq\,c g$. Here $c$ is a generic constant which does not depend on the crucial quantities. If necessary we specify particular dependencies. We write $f\approx g$ if $f\lesssim g$ and $g\lesssim f$.
We do not distinguish in the notation for the function spaces between scalar- and vector-valued functions. However, vector-valued functions will usually be denoted in bold case.

As usual we denote by $B_r(x)\in\R^n$ the ball around $x\in\R^n$ with radius $r>0$. Similarly, $B_r'(x')\in\R^{n-1}$ is the ball around $x'\in\R^{n-1}$ and 
$B_r^+(x):=B_r(x)\cap\mathbb H\in\R^n$ is the half ball around $x\in\R^n$, where $\mathbb H:=\{x=(x',x_n)\in\R^n:\,x_n>0\}$ is the half space.

\subsection{Function spaces}
Let $\mathcal O\subset\R^n$, $n\geq 1$, be open.
Function spaces of continuous or $\alpha$-H\"older-continuous functions, $\alpha\in(0,1)$,
 are denoted by $C({\mathcal O})$ or $C^{0,\alpha}({\mathcal O})$ respectively. Similarly, we write $C^1({\mathcal O})$ and $C^{1,\alpha}({\mathcal O})$ for spaces of functions which are ($\alpha$-H\"older) continuously differentiable.
We denote as usual by $L^p(\mathcal O)$ and $W^{k,p}(\mathcal O)$ for $p\in[1,\infty]$ and $k\in\mathbb N$ Lebesgue and Sobolev spaces over $\mathcal O$. For a bounded domain $\mathcal O$ the space $L^p_\perp(\mathcal O)$ denotes the subspace of  $L^p(\mathcal O)$ of functions with zero mean, that is $(f)_{\mathcal O}:=\dashint_{\mathcal O}f\dx:=\mathscr L^n(\mathcal O)^{-1}\int_{\mathcal O}f\dx=0$. %For unbounded domains we require that $f\in L^p(\mathcal O)\cap L^1(\mathcal O)$ and $\int_{\mathcal O}f\dx=0$.
 We denote by $W^{k,p}_0(\mathcal O)$ the closure of the smooth and compactly supported functions in $W^{k,p}(\mathcal O)$. If $\partial\mathcal O$ is regular enough, this coincides with the functions vanishing $\mathcal H^{n-1}$ -a.a. on $\partial\mathcal O$.  We also denote by $W^{-k,p}(\mathcal O)$ the dual of $W^{k,p}_0(\mathcal O)$.
  Finally, we consider subspaces
$W^{1,p}_{\Div}(\mathcal O)$ and $W^{1,p}_{0,\Div}(\mathcal O)$ of divergence-free vector fields which are defined accordingly. Last we introduce for unbounded domains $\mathcal O$ the homogeneous Sobolev spaces $\mathcal D^{k,p}(\mathcal O)$ as the set of all locally $p$-integrable functions with finite $\|\nabla^k\cdot\|_{L^p(\mathcal O)}$-semi norm. Furthermore, $\mathcal D^{k,p}_0(\mathcal O)$ is defined as the closure of $C_c^\infty(\mathcal O)$ with respect to $\|\nabla^k\cdot\|_{L^p(\mathcal O)}$
along with its solenoidal variant $\mathcal D_{0,\Div}^{k,p}(\mathcal O)$, where we take the closure of
$C_{c,\Div}^\infty(\mathcal O)$ instead. Here $C_{c}^\infty(\mathcal O)$ and $C_{c,\Div}^\infty(\mathcal O)$ denote the spaces of smooth and compactly supported (solenoidal) functions.

\subsection{Campanato spaces}\label{subsec:cs}
In this subsection we give a brief introcution to Campanto spaces \cite{CaA,CaB} and collect some of their useful properties to be used later.

Let $\mathcal{O}$ be an open set in $\mathbb R^n$, which is fat. This means that it is assumed to have the
property that there exists a constant $C$ such that
$$|B \cap \mathcal{O}| \geq C |B|$$
for every ball $B$, centered at a point on $\partial \mathcal{O}$.
 We define for a ball $B$ and
$g\in L^q(\mathcal{O} )$, $q\geq1$,
\begin{align*}
%  ( \mathcal M^{\sharp,q}_{\mathcal{O}, R} g)(x)&=\sup_{r\leq R}\bigg(\dashint_{B_r\cap \mathcal{O}} \abs{g-(g)_{B\cap \mathcal{O}}}^q\dx\bigg)^\frac1q,
%  \\
  (\mathcal M^{\sharp,q}_{\mathcal{O}} g)(x)&:=\sup_{r>0}\bigg(\dashint_{B_r(x)\cap \mathcal{O}} \abs{g-(g)_{B_r(x)\cap \mathcal{O}}}^q\dx\bigg)^\frac1q.
\end{align*}
We set further % $\mathcal M^\sharp_{\mathcal{O}, R} =\mathcal M^{\sharp,1}_{\mathcal{O}, R}$ and 
$\mathcal M^{\sharp}_\mathcal{O}=\mathcal M^{\sharp,1}_{\mathcal{O}}$.
%Finally, we define the Hardy Littlewood maximal operators by
%\begin{align}
%\begin{aligned} 
%  (M^{q}_{\mathcal{O}, R} g)(x)&=\sup_{r\leq R}\bigg(\dashint_{B_r\cap \mathcal{O}} \abs{g}^q\dx\bigg)^\frac1q,
%  \\
%  (M^{q}_{\mathcal{O}} g)(x)&:=\sup_{r>0}\bigg(\dashint_{B_r\cap \mathcal{O}} \abs{g}^q\dx\bigg)^\frac1q.
%\end{aligned}
%\end{align}
The space $\BMO(\mathcal O)$ of functions of bounded mean oscillations is defined via the
following semi norm
\begin{align*}
  \norm{g}_{\BMO(\mathcal{O})}=\sup_{B_r}
  \dashint_{B_r\cap \mathcal{O}}\abs{g-(g)_{B_r\cap \mathcal{O}}}\dx 
\end{align*}
saying that $g\in \BMO(\mathcal{O} )$, whenever its semi-norm is
bounded. Therefore $g\in \BMO(\mathcal{O})$ if and only if $\mathcal M^\sharp
g\in L^\infty(\mathcal{O})$. By H\"older's inequality and the classical John-Nirenberg estimate \cite{JN} an equivalent characterisation is given by
\begin{align*}
  \norm{g}_{\BMO(\mathcal{O})}\approx  \sup_{B_r}\bigg(\dashint_{B_r\cap \mathcal{O}} \abs{g-(g)_{B_r\cap \mathcal{O}}}^q\dx\bigg)^\frac1q
\end{align*}
for any $q\in(1,\infty)$. In view of our application to \eqref{3} our preferred choice here is $q=2$.
\\
More generally, for a non-decreasing function $\omega \,:\,
(0,\infty) \to (0,\infty)$ we define
\begin{align}
\begin{aligned} 
%  (\mathcal M^{\sharp,q}_{\omega,\mathcal{O}, R} g)(x)&=\sup_{r\leq R}\frac{1}{\omega(r)}\bigg(\dashint_{B_r\cap \mathcal{O}} \abs{g-(g)_{B_r\cap \mathcal{O}}}^q\dx\bigg)^\frac1q,
%  \\
  (\mathcal M^{\sharp,q}_{\omega,\mathcal{O}} g)(x)&:=\sup_{r>0}\frac{1}{\omega(r)}\bigg(\dashint_{B_r(x)\cap \mathcal{O}} \abs{g-(g)_{B_r(x)\cap \mathcal{O}}}^q\dx\bigg)^\frac1q.
\end{aligned}
\end{align}
The Campanato semi-norm
associated with $\omega$ is given by
\begin{align*}
  \norm{g}_{\mathcal L^\omega(\mathcal{O})}= \sup_{B_r}\frac{1}{\omega(r)}\bigg(\dashint_{B_r\cap \mathcal{O}} \abs{g-(g)_{B_r\cap \mathcal{O}}}^2\dx\bigg)^{\frac{1}{2}}.
\end{align*}
 Finally, we also consider higher order Campanato spaces $W^1\mathcal L^\omega$ ($W^2\mathcal L^\omega$) which are defined by requiring that a function and its (first and second order) weak derivative(s) belong to $\mathcal L^\omega$.

We denote by $C^{0, \omega}(\mathcal{O})$ the space of  functions $f$ in $\mathcal{O}$ endowed with the semi-norm
\begin{equation}\label{C0om}
\|f \|_{C^{0, \omega }(\mathcal{O})}=
\sup_{
\begin{tiny}
 \begin{array}{c}{
    x, y \in \mathcal{O}} \\
x \neq y
 \end{array}
  \end{tiny}
}
% \sup_{x, y \in \mathcal{O}}
\frac{|f (x) - f(y)|}{\omega (|x-y|)}\,.
\end{equation}
Plainly, if $\omega (0)=0$, then $C^{0, \omega}(\mathcal{O})$ is a space of uniformly continuous functions in $\mathcal{O}$,  with modulus of continuity not exceeding $\omega$. 
%If $\omega (r)= r^\beta$ for some $\beta \in (0, 1]$, then   $C^{0, \omega (\cdot)}(\mathcal{O})$ coincides with the space of H\"older continuous functions with exponent $\beta$, that will simply 
 It is shown in \cite{Sp} that, if $\mathcal{O}$ is a bounded Lipschitz domain, and the parameter function $\omega$ fulfils the condition
\begin{equation}\label{dini} \int _0^{\textcolor{blue}{\cdot}} \frac{\omega (r)}r\, {\rm d} r < \infty,
\end{equation}
 then 
\begin{equation}\label{spanne1}
\mathcal L ^\omega(\mathcal{O}) \to  {\color{black} C^{0, \underline \omega}(\mathcal{O})},
\end{equation}
where $\underline \omega$ is defined by
\begin{equation}\label{dini1}
%\varpi 
\underline \omega (r) = \int _0^r \frac{\omega (\rho)}\rho\,\mathrm{d}\rho \, \quad
\hbox{for $r \geq 0$.}
\end{equation}
Note that, as shown in  \cite{Sp}, if the function $\tfrac{\omega (r)}r$ is non-increasing, condition \eqref{dini} is necessary even for an embedding of the space $\mathcal L ^\omega(\mathcal{O})$ into
$L^\infty (\mathcal{O})$. Also, the function $\underline \omega$ is optimal in \eqref{spanne1}.

Based on these observations one easily sees in the case $\omega(r)=r^\beta$ that the spaces $\mathcal L^\omega(\mathcal{O})$ and $C^{0,\omega}$ coincide and agree with the space of $\beta$-H\"older continuous functions by $C^{0,\beta}(\mathcal{O})$. For general weights we require that there is $\beta_0\in(0,1)$ such that
\begin{equation} \label{eq:omega condition}
\omega(r) \leq c_\omega \theta^{-\beta_0} \omega(\theta r) \qquad \hbox{for
 $\theta  \in (0,1)$,}
\end{equation}
for some  constant $c_\omega$.

Finally, we consider multipliers on $\mathcal L^\omega(\mathcal{O})$. 
The space $\mathscr M^\omega(\mathcal{O})$ of Campanto multipliers is defined as the measurable functions for which the $\mathscr M^\omega(\mathcal{O})$-norm given by
\begin{align}\label{eq:SoMo'}
\|v\|_{\mathscr M^\omega(\mathcal{O})}:=\sup_{w:\,\|w\|_{\mathcal L^\omega(\mathcal{O})}+|(w)_\mathcal{O}|=1}\bigg(\|v \,w\|_{\mathcal L^\omega(\mathcal{O})}+|(vw)_\mathcal{O}|\bigg)
\end{align}
is finite. Multipliers on $\mathcal L^\omega(\mathcal{O})$ are fully characterised in \cite[Theorem 1.3]{Na}\footnote{In \cite{Na} the more general case of multipliers between Campanato spaces with possibly different weights are analysed.} in terms of the function
\begin{align}\label{eq:sigma}
\sigma(r):=\omega(r)\bigg(\int_\rho^1\frac{\omega(\rho)}{\rho}\,\dd\rho\bigg)^{-1}.
\end{align}
In particular, it is show that $\mathscr M^\omega(\mathcal{O})=\mathcal L^\sigma(\mathcal{O})$ and it holds
\begin{align}\label{eq:cm}
\|v\|_{\mathscr M^\omega(\mathcal{O})}\approx \|v\|_{\mathcal L^\sigma(\mathcal{O})}+\|v\|_{L^1(\mathcal{O})}.
\end{align}
This characterisation is crucial in the proof of Theorem \ref{thm:stokessteady}. 

If the functions in question are essentially bounded, we make use of the following trivial multiplier estimate: Suppose that $v,w\in L^\infty(\mathcal O)$. Then it holds
\begin{align*}
\dashint_{B_r\cap \mathcal O}|vw-(vw)_{B_r\cap\mathcal O}|^2\dx&\lesssim \dashint_{B_r\cap \mathcal O}|vw-v(w)_{B_r\cap\mathcal O}|^2\dx+\dashint_{B_r\cap \mathcal O}|v(w)_{B_r\cap\mathcal O}-(v)_{B_r\mathcal O}(w)_{B_r\cap \mathcal O}|^2\dx\\
&+\dashint_{B_r\cap \mathcal O}|(v)_{B_r\cap\mathcal O}(w)_{B_r\cap\mathcal O}-(vw)_{B_r\cap\mathcal O}|^2\dx\\
&\lesssim \|v\|^2_{L^\infty(\mathcal O)}\dashint_{B_r\cap \mathcal O}|w-(w)_{B_r\cap\mathcal O}|^2\dx+\|w\|_{L^\infty(\mathcal O)}^2\dashint_{B_r\cap \mathcal O}|v-(v)_{B_r\cap\mathcal O}|^2\dx
\end{align*}
and, by taking the supremum in $r$,
\begin{align} \label{eq:cm+}
\|uv\|_{\mathcal L^\omega(\mathcal O)}\lesssim \|v\|_{L^\infty(\mathcal O)}\| 
w\|_{\mathcal L^\omega(\mathcal O)}+ \|w\|_{L^\infty(\mathcal O)}\| 
v\|_{\mathcal L^\omega(\mathcal O)}
\end{align}
for all $v,w\in \mathcal L^\omega(\mathcal O)\cap L^\infty(\mathcal O)$. We will benefit from \eqref{eq:cm+} in the proof of Theorem \ref{thm:stokessteady2}.

\subsection{Parametrisation of domains}\label{sec:para}
In this section we describe how to parametrise the boundary of a domain by local charts.
We follow the presentation from \cite{Br} (see also \cite{Br2}).
 Let $\Omega\subset\R^n$ be a bounded open set.
We assume that $\partial{\Omega}$ can be covered by a finite
number of open sets $\mathcal U^1,\dots,\mathcal U^\ell$ for some $\ell\in\mathbb N$, such that
the following holds. For each $j\in\{1,\dots,\ell\}$ there is a reference point
$y^j\in\R^n$ and a local coordinate system $\{e^j_1,\dots,e_n^j\}$ (which we assume
to be orthonormal and set $\mathcal Q_j=(e_1^j|\dots |e_n^j)\in\mathbb R^{n\times n}$, i.e., the matrix with column vectors $e_1^j,\dots ,e_n^j$), a function
$\varphi_j:\mathbb R^{n-1}\rightarrow\mathbb R$
%$$\varphi_j:\big(\mathbb H^j=y^j+\mathrm{span}\{e^j_1,\dots,e_{n-1}^j\}\big)\cap B_{r_j}(y_j)\rightarrow\mathbb R$$
%and a function $\Psi_j:\mathbb H^j\cap B_{r_j}(y^j)\rightarrow\mathbb R$ (
and $r_j>0$
with the following properties:
\begin{enumerate}[label={\bf (A\arabic{*})}]
\item\label{A1} There is $h_j>0$ such that
$$\mathcal U^j=\{x=\mathcal Q_jz+y^j\in\mathbb R^n:\,z=(z',z_n)\in\R^n,\,|z'|<r_j,\,
|z_n-\varphi_j(z')|<h_j\}.$$
\item\label{A2} For $x\in\mathcal U^j$ we have with $z=\mathcal Q_j^\top(x-y^j)$
\begin{itemize}
\item $x\in\partial{\Omega}$ if and only if $z_n=\varphi_j(z')$;
\item $x\in{\Omega}$ if and only if $0<z_n-\varphi_j(z')<h_j$;
\item $x\notin{\Omega}$ if and only if $0>z_n-\varphi_j(z')>-h_j$.
\end{itemize}
\item\label{A3} We have that
$$\partial{\Omega}\subset \bigcup_{j=1}^\ell\mathcal U^j.$$
\end{enumerate}
% Let $x_0\in\partial{\Omega}$ be boundary point. For simplicity of presentation we can assume by translation and rotation that $x_0=0$ and that the outer normal at~$x_0$ is pointing in the negative $x_n$-direction. Now, let $\phi:\mathbb R^{n-1}\rightarrow\mathbb R$ be the local map describing the boundary on the neighborhood~$U$ of~$0$, that is we have 
In other words, for any $x_0\in\partial{\Omega}$ there is a neighbourhood $U$ of $x_0$ and a function $\varphi:\mathbb R^{n-1}\rightarrow\mathbb R$ such that after translation and rotation\footnote{By translation via $y_j$ and rotation via $\mathcal Q_j$ we can assume that $x_0=0$ and that the outer normal at~$x_0$ is pointing in the negative $x_n$-direction.}
 \begin{align}\label{eq:3009}
 U \cap {\Omega} = U \cap G,\quad G = \set{(x',x_n)\in \R^n \,:\, x' \in \R^{n-1}, x_n > \varphi(x')}.
 \end{align}
% for some $R>0$. Clearly, $\varphi$ and $R$ depend on $x_0$.

The regularity of $\partial{\Omega}$ can now be described by means of the a parametrisation functions $\varphi_1,\dots,\varphi_\ell$.
 \begin{definition}\label{def:besovboundary}
 Let ${\Omega}\subset\R^n$ be a bounded domain and $\sigma:(0,\infty)\rightarrow(0,\infty)$ non-decreasing. We say that $\partial{\Omega}$ belongs to the class $W^1\mathcal L^\sigma$ if there is $\ell\in\mathbb N$ and functions $\varphi_1,\dots,\varphi_\ell\in W^1\mathcal L^\sigma$ satisfying \ref{A1}--\ref{A3}.
% and the functions $\bfPhi_1,\dots,\bfPhi_\ell$ belong to the class $W^1\mathcal L^\sigma$.
 \end{definition}
 % Let $L$ denote the local Lipschitz constant of~$\phi$.
We are also interested in domains for which the
$W^{1}\mathcal L^\sigma$ semi-norm is sufficiently small: We say that a domain belongs to the class $W^1\mathcal L^\sigma(\delta)$ if
$\|\nabla\varphi_j\|_{\mathcal L^\sigma}$, $j=1,\dots,\ell$, does not exceed $\delta$. Note that
if $\partial{\Omega}\in W^1\mathcal L^{\tilde \sigma}$ for a parameter function $\tilde\sigma$  which decays faster to zero as the parameter function $\sigma$ one immediately obtains 
 $W^1\mathcal L^\sigma(\delta)$ for any given $\delta$ by introducing sufficiently many charts and arguing locally.
\begin{remark}

Recalling \eqref{spanne1} and choosing $\omega(r)=r^\beta$ for $0<\beta<1$ there our definition includes $C^{1,\beta}$-domains as a special case. Also note that various other characterisations can be obtained by a variation of Definition \ref{def:besovboundary}. For instance, we obtain Lipschitz domains ($C^1$-domains) by requiring $\varphi_1,\dots,\varphi_\ell\in W^{1,\infty}$ ($\varphi_1,\dots,\varphi_\ell\in C^{1}$) instead. Analogously, one can introduce domains with small local Lipschitz constants.
\end{remark}

Given a local chart $\varphi$ as in \eqref{eq:3009} we define $\bfPhi: \R^n\rightarrow\R^n$ as 
\begin{align}\label{PSI}
\bfPhi(x',x_n)=(x',x_n-\varphi(x')) \quad \hbox{for $(x',x_n) \in \R^n$.}
\end{align}
Now, if $\varphi:B_r'(0)\rightarrow\R$ describes a part of $\partial\Omega$ then
$\bfPhi: \overline{\Omega}\cap B_r(0)\to  B_h^+(0)$
 %where we abbreviated $(x_1,\ldots, x_n):= (x',x_n)$.
together with $\bfPhi(\partial\Omega\cap B_r(0))\subset \set{(x',x_n):x_n=0}$ and $\bfPhi(0)=0$. 
Moreover,  the function $\bfPhi : \overline{\Omega}\cap B_r(0) \to \bfPhi(\overline{\Omega}\cap B_r(0))$ is invertible, with inverse
$\bfPsi :   \bfPhi(\overline{\Omega}\cap B_r(0)) \to \overline{\Omega}\cap B_r(0)$. We suppose that $\varphi$ is Lipschitz continuous as this is a necessary assumption in all our statements. Hence also $\bfPhi$ and $\bfPsi$ are Lipschitz continuous.
If we define $\bfJ : \overline{\Omega}\cap B_r(0)\to \setR ^{n\times n}$ as
\begin{equation}\label{J} \bfJ(x) = \nabla \bfPhi (x) \quad \hbox{for $x \in \overline{\Omega}\cap B_r(0)$,}
\end{equation}
 then $\det \bfJ (x)=1$ for $x \in \overline{\Omega}\cap B_{\textcolor{blue}{r}}(0)$ and  $\det\bfJ^{-1} (y)=1$ for $y \in  \bfPhi(\overline {\Omega}\cap B_r(0)) $. Owing to the Lipschitz continuity of $\bfPhi$ and $\bfPsi$,  there exist  constants 
$0<\lambda \leq 1 \leq \Lambda$
 such that 
 \begin{align}\label{inclusion1}
  B_{\lambda s}^+(0)\subset \bfPhi(\Omega\cap B_s(0)) \subset B_{\Lambda s}^+(0)
\end{align}
if $0<s \leq r$,
and
 \begin{align}\label{inclusion2}
 &  \Omega  \cap B_{\frac s\Lambda }(0)\subset \bfPsi(B_s^+(0)) \subset  \Omega \cap B_{\frac{s}{\lambda}}(0)
 \end{align}
if   $B_s^+(0) \subset \bfPhi(\overline \Omega\cap B_r(0))$. 

Consequently,
for $u\in \mathcal L^\omega(B_{s\Lambda}^+(0))$  we have
\begin{align}\label{lem:9.4.1}
\|u\circ\bfPsi\|_{\mathcal L^\omega( \Omega\cap B_s(0))}\lesssim \|u\|_{\mathcal L^\omega(B_{s\Lambda}^+(0))}.
\end{align}
%with constant depending on $\lambda$ and $\Lambda$ from \eqref{inclusion1}.

\section{The Sokes equations in the half space}
\label{sec:half}
We consider the Sokes equation
\begin{align}\label{eq:Stokeshalf}
\Delta \bfu-\nabla\pi=-\bff=-\Div\bfF,\quad\Div\bfu=h,\quad\bfu|_{\partial{\mathbb H}}=0,
\end{align}
 in the half space
$\mathbb H:=\{x=(x',x_n)\in\R^n:\,x_n>0\}$. Here the functions $\bff:\mathbb H \rightarrow\R^n$ respectively $\bfF:\mathbb H\rightarrow\R^{n\times n}$ and $h:\mathbb H\rightarrow\R$ are given. 

\subsection{Decay estimates}
A solution $(\bfu,\pi)$ to problem \eqref{eq:Stokeshalf} in divergence-form exists in the homogeneous space $\mathcal D^{1,2}_{0,\Div}(\mathbb H)\times L^2(\mathbb H)$ provided $\bfF,h\in L^2(\mathbb H)$, see \cite[Chapter IV, Theorem IV 3.3]{Ga}. The velocity field is unique and the pressure function is unique up to a constant. The following theorem gives a decay estimate for boundary points.
\begin{theorem}\label{thm:decay}
Suppose that we have $\bfF\in L^2(\mathbb H)$ and that
 $(\bfu,\pi)\in \mathcal D^{1,2}_{0,\Div}(\mathbb H)\times L^2(\mathbb H)$ is the solution to \eqref{eq:Stokeshalf} with $h=0$. Then we have
\begin{align}\label{est:halfu}
\begin{aligned}
\dashint_{B^+_{\theta r}(x_0)}&|\nabla\bfu-(\partial_n\bfu\otimes e_n)_{B^+_{\theta r}(x_0)}|^2\dx+\dashint_{B^+_{\theta r}(x_0)}|\pi-(\pi)_{B^+_{\theta r}(x_0)}|^2\dx\\&\lesssim \theta^2 \dashint_{B^+_{r}(x_0)}|\nabla\bfu-(\partial_n\bfu\otimes e_n)_{B^+_{ r}(x_0)}|^2\dx
%+\theta^2 \dashint_{B^+_{r}(x_0)}|\pi-(\pi)_{B^+_{r}(x_0)}|^2\dx\\&
+c_\theta \dashint_{B^+_{r}(x_0)}|\bfF-(\bfF)_{B^+_{ r}(x_0)}|^2\dx,
\end{aligned}\end{align}
for all $x_0\in \partial\mathbb H$, all $r>0$ and all $\theta\in(0,1)$.
\end{theorem}
\begin{proof}
Without loss of generality let us assume that $x_0=0$ and set $B_r^+:=B_r^+(0)$. 
We first consider a homogenous auxiliary problem:
Let $(\bfh,\vartheta)\in W^{1,2}_{\Div}(B_r^+)\times L^2(B_r^+)$ be the unique solution\footnote{See \cite[Chapter IV.6]{Ga} for its existence.} of the problem
\begin{align}\label{eq:Stokeshalfhomo}
\Delta \bfh-\nabla\vartheta=0,\quad\Div\bfh=0,\quad\bfh|_{B_r^+\cap \partial{\mathbb H}}=0,\quad\bfh|_{\partial B_r^+\cap {\mathbb H}}=\bfu,
\end{align}
in $B_r^+$. Since the pressure $\vartheta$ is only unique up to a constant we may assume that $(\vartheta)_{B^+_{ r}}=(\pi)_{B^+_{ r}}$
 Testing the equation for $\bfv:=\bfu-\bfh$ by
$\bfv$ and noticing that $\Div\bfF=\Div(\bfF-(\bfF)_{B^+_{ r}})$ yields
\begin{align}\label{eq:comparison}
\dashint_{B^+_{r}}|\nabla\bfu-\nabla\bfh|^2\dx\lesssim \dashint_{B^+_{r}}|\bfF-(\bfF)_{B^+_{ r}}|^2\dx.
\end{align}
With the aid of the negative norm theorem from \cite{Ne} we also get
\begin{align}\label{eq:comparison2}
\begin{aligned}
\dashint_{B^+_{r}}|\pi-\vartheta|^2\dx&\lesssim r^{-n}\|\pi-\vartheta-(\pi-\vartheta)_{B_r^+}\|_{L^2(B_r^+)}^2\lesssim r^{-n}\|\nabla(\pi-\vartheta)\|_{W^{-1,2}(B_r^+)}^2\\&= r^{-n}\|\Div(\nabla\bfu-\nabla\bfh+\bfF-(\bfF)_{B_r^+})\|_{W^{-1,2}(B_r^+)}^2\\
&\lesssim \dashint_{B^+_{r}}|\nabla\bfu-\nabla\bfh|^2\dx+ \dashint_{B^+_{r}}|\bfF-(\bfF)_{B^+_{ r}}|^2\dx\\
&\lesssim \dashint_{B^+_{r}}|\bfF-(\bfF)_{B^+_{ r}}|^2\dx.
\end{aligned}
\end{align}
% It is easy to see that the pressure
%$\vartheta$ satisfies
%\begin{align*}
%\Delta\vartheta=0,\quad \partial_n\vartheta|_{\partial B_r^+\cap {\mathbb H}}=0.
%\end{align*}
%Hence it satisfies for $k\in\N$
%\begin{align} \label{0q:0303a}
%\dashint_{B^+_{\theta r}}|\vartheta-(\vartheta)_{B^+_{\theta r}}|^2\dx&\lesssim \theta^2 \dashint_{B^+_{r}}|\vartheta-(\vartheta)_{B^+_{r}}|^2\dx,\\ \int_{B_{s}^+}|\nabla^k\vartheta|^2\dx&\lesssim \int_{B_{r}^+}\frac{|\nabla^{k-1}(\vartheta-(\vartheta)_{B_r^+})|^2}{(r-s)^2}\dx, \label{0q:0303b}
%\end{align}
%as can be proved by
Now we consider for $\gamma\in\{1,\dots, n-1\}$, $0<s<r$, and $\eta\in C_c^\infty (B_r)$ with $\eta=1$ in $B_{s}$ and $|\nabla\eta|\lesssim (r-s)^{-1}$ the test-function
\begin{align*}
\bfvarphi=\partial_\gamma\big(\eta^2\partial_\gamma\bfh-\mathrm{Bog}_{B_r^+}(\nabla\eta^2\cdot\partial_\gamma\bfh)\big).
\end{align*}
Here $\mathrm{Bog}_{B_r^+}$ denotes the Bogoskii operator on ${B_r^+}$, see \cite[Section III.3]{Ga} for its properties.
%where $\bfP==(\tilde\bfP,\bfP_n)=(\bfP_1,\dots,\bfP_n)\in\R^{n\times n}$ is arbitrary.
This yields
\begin{align*}
\int_{B_s^+}|\partial_\gamma\nabla\bfh|^2\dx&\leq\int_{B_r^+}\eta^2|\partial_\gamma\nabla\bfh|^2\dx\\&=-\int_{B_r^+}\partial_\gamma\nabla\bfh:\nabla\eta^2\otimes\partial_\gamma\bfh\dx+\int_{B_r^+}\partial_\gamma\nabla\bfh:\nabla\mathrm{Bog}_{B_r^+}(\nabla\eta^2\cdot\partial_\gamma\bfh)\dx\\
&\leq\tfrac{1}{2}\int_{B_r^+}|\partial_\gamma\nabla\bfh|^2\dx+\tfrac{c}{(r-s)^2}\int_{B_r^+}|\partial_\gamma\bfh|^2\dx,
\end{align*}
where we used well-known continuity properties of the Bogovskii operator in the last step. By the iteration lemma from \cite[Lemma 3.1, p. 161]{Gi} we can remove the first term on the right-hand side obtaining
\begin{align}\label{eq:0303c}
\int_{B_{s}^+}|\tilde\nabla\nabla\bfh|^2\dx\lesssim \int_{B_{r}^+}\frac{|\tilde\nabla\bfh|^2}{(r-s)^2}\dx.
\end{align}
Since $\Div \bfh=0$ the only missing part of $\nabla^2\bfh$ is $\partial_n^2\tilde\bfh$, where we denote $\tilde\bfh=(h^1,\dots, h^{n-1})$, which satisfies the equation
\begin{align}\label{eq:1703}
\partial_n^2\tilde\bfh=-\tilde \Delta\tilde\bfh+\tilde{\nabla}\vartheta,\quad \tilde\nabla:=(\partial_1,\dots,\partial_{n-1}),\quad\tilde\Delta =\sum_{j=1}^{n-1}\partial_j^2.
\end{align}
Hence we are due to estimate $\tilde{\nabla}\vartheta$. Here we implement a trick from \cite{BKR}, differentiate the equation for $\bfh$ with respect to $\gamma\in\{1,\dots,n-1\}$ and multiply it by
\begin{align*}
\mathrm{Bog}_{B_\varrho^+}\big(\eta^2\partial_\gamma\vartheta-(\eta^2\partial_\gamma\vartheta)_{B_\varrho^+}\big)\in W^{1,2}_{0}(B_\varrho^+),
\end{align*} 
where we now assume that $\eta\in C^\infty_c(B_\varrho)$ with $\varrho<\frac{s+r}{2}$.
We obtain
\begin{align*}
\int_{B_\varrho^+}\eta^2|\partial_\gamma\vartheta|^2\dx&=\int_{B_\varrho^+}\partial_\gamma\nabla\bfh:\nabla \mathrm{Bog}_{B_\varrho^+}\big(\eta^2\partial_\gamma\vartheta-(\eta^2\partial_\gamma\vartheta)_{B_\varrho^+}\big)\dx\\
&+\int_{B_\varrho^+}\partial_\gamma\vartheta\dx\dashint_{B_\varrho^+}\eta^2\partial_\gamma\vartheta\dx.
\end{align*}
The first term on the right-hand side is clearly bounded by
\begin{align*}
\kappa&
\int_{B_\varrho^+}
\big|\nabla \mathrm{Bog}_{B_\varrho^+}\big(\eta^2\partial_\gamma\vartheta-(\eta^2\partial_\gamma\vartheta)_{B_\varrho^+}\big)\big|^2\dx+c(\kappa)\int_{B_\varrho^+}|\partial_\gamma\nabla\bfh|^2\dx\\
&\qquad\qquad\lesssim \kappa
\int_{B_\varrho^+}
\big|\big(\eta^2\partial_\gamma\vartheta-(\eta^2\partial_\gamma\vartheta)_{B_\varrho^+}\big)\big|^2\dx+c(\kappa)\int_{B_\varrho^+}|\partial_\gamma\nabla\bfh|^2\dx\\
&\qquad\qquad\lesssim \kappa
\int_{B_\varrho^+}
\eta^2|\partial_\gamma\vartheta|^2\dx+c(\kappa)\int_{B_\varrho^+}|\partial_\gamma\nabla\bfh|^2\dx
\end{align*}
using Young's inequality for $\kappa>0$ and the continuity properties of the Bogovskii operator. The first term can be absorbed for $\kappa$ sufficiently small.
Since $\eta=0$ on $\partial B_\varrho^+\cap \mathbb H$ and $\gamma\neq n$ we can integrate by parts in the second term obtaining
\begin{align*}
 \int_{B_\varrho^+}\partial_\gamma\vartheta\dx\dashint_{B_\varrho^+}\eta^2\partial_\gamma\vartheta\dx&=- \int_{B_\varrho^+}\partial_\gamma\vartheta\dx\dashint_{B_\varrho^+}\partial_\gamma\eta^2(\vartheta-(\vartheta)_{B_\varrho^+})\dx\\
&\leq \tfrac{1}{2}\int_{B_\varrho^+}
|\partial_\gamma\vartheta|^2\dx+c\int_{B_\varrho^+}\frac{|\vartheta-(\vartheta)_{B_\varrho^+}|^2}{(\varrho-s)^2}\dx.
\end{align*}
The first term on the right-hand side can be eliminated again with the help of the iteration lemma from \cite[Lemma 3.1, p. 161]{Gi}.
Combining everything, choosing $\kappa$ small enough and using \eqref{eq:0303c}
proves
\begin{align}\label{eq:0303c'}
\int_{B_{s}^+}|\tilde\nabla\vartheta|^2\dx\lesssim \int_{B_{r}^+}\frac{|\tilde\nabla\bfh|^2}{(r-s)^2}\dx+\int_{B_{r}^+}\frac{|\vartheta-(\vartheta)_{B_r^+}|^2}{(r-s)^2}\dx.
\end{align}
Recalling \eqref{eq:1703} this allows to control $\nabla^2\bfh$ in the sense that even
\begin{align*}
\int_{B_{s}^+}|\tilde \nabla\vartheta|^2\dx+\int_{B_{s}^+}|\nabla^2\bfh|^2\lesssim \int_{B_{r}^+}\frac{|\tilde\nabla\bfh|^2}{(r-s)^2}\dx+\int_{B_{r}^+}\frac{|\vartheta-(\vartheta)_{B_r^+}|^2}{(r-s)^2}\dx.
\end{align*}
Finally, we have $\partial_n\vartheta=\Delta h^n$ and thus
\begin{align}\label{eq:cacc2}
\int_{B_{s}^+}|\nabla\vartheta|^2\dx+\int_{B_{s}^+}|\nabla^2\bfh|^2\lesssim \int_{B_{r}^+}\frac{|\tilde\nabla\bfh|^2}{(r-s)^2}\dx+\int_{B_{r}^+}\frac{|\vartheta-(\vartheta)_{B_r^+}|^2}{(r-s)^2}\dx.
\end{align}
%In order to obtain a corresponding estimate for $\nabla^2\bfh$ it remains to control $\partial_n^2\bfh$, which satisfies
%\begin{align*}
%\partial_n^2\bfh=-\tilde \Delta\bfh+\nabla\vartheta,\quad\tilde\Delta =\sum_{j=1}^{n-1}\partial_j^2.
%\end{align*}
%Thus we have
%\begin{align*}
%\int_{B_{s}^+}|\partial_n^2\bfh|^2\dx&\lesssim \int_{B_{r}^+}\frac{|\partial_\gamma\bfh|^2}{(r-s)^2}\dx+\int_{B_{s}^+}|\nabla\vartheta|^2\dx\\
%&\lesssim \int_{B_{r}^+}\frac{|\partial_\gamma\bfh|^2}{(r-s)^2}\dx+\int_{B_{r}^+}\frac{|\vartheta-(\vartheta)_{B_r^+}|^2}{(r-s)^2}\dx.
%\end{align*}
%By the negative norm theorem from \cite{Ne}\footnote{The latter can be derived from the Bogoskii operator which is continuous on star-shaped domains, cf. \cite[Chapter 3]{Ga}.} we have
%\begin{align*}
%\int_{B_{r}^+}|\vartheta-(\vartheta)_{B_r^+}|^2\dx\lesssim \|\nabla\vartheta\|_{W^{-1,2}(B_r^+)}= \|\Delta\bfh\|_{W^{-1,2}(B_r^+)}\lesssim
%\int_{B_{r}^+}|\nabla\bfh-(\nabla\bfh)_{B_r^+}|^2\dx.
%\end{align*}
%Combining the above we have shown
%\begin{align}\label{eq:cacc2}
%\int_{B_{s}^+}|\nabla^2\bfh|^2\dx\lesssim \int_{B_{r}^+}\frac{|\nabla\bfh-(\partial_n\bfh\otimes e_3)_{B_r^+}|^2}{(r-s)^2}\dx
%\end{align}
%using minimality of the mean value.

Now are now going to prove higher order variants of
\eqref{eq:cacc2}. First of all, \eqref{eq:0303c}
extends to higher order tangential derivatives, that is we have
\begin{align}\label{eq:0303d}
\int_{B_{s}^+}|\tilde\nabla^k\nabla\bfh|^2\dx\lesssim \int_{B_{r}^+}\frac{|\tilde\nabla^k\bfh|^2}{(r-s)^2}\dx
\end{align}
for $k\in\N$. Similarly, we obtain
\begin{align}\label{eq:0303e}
\int_{B_{s}^+}|\tilde\nabla^k\vartheta|^2\dx\lesssim \int_{B_{r}^+}\frac{|\tilde\nabla^k\bfh|^2}{(r-s)^2}\dx+\int_{B_{r}^+}\frac{|\tilde\nabla^{k-1}\vartheta-(\tilde\nabla^{k-1}\vartheta)_{B_{r}^+}|^2}{(r-s)^2}\dx.
\end{align}

To obtain an estimate for $\nabla^k\bfh$ is rather tedious and we only comment
on the case $k=2$. The general case follows in the same manner. If $k=2$ it is sufficient to control $\partial_n^2\tilde\nabla\tilde\bfh$, for which
 we use again the equation and write
\begin{align*}
\partial_n^2\tilde\nabla\tilde\bfh=-\tilde \Delta\tilde\nabla\tilde\bfh+\tilde\nabla^2\vartheta.
\end{align*}
%We can now use
%\eqref{0q:0303b} to estimate the pressure
%such that \eqref{eq:0303d} yields
%\begin{align}\label{eq:cacc3}
%\int_{B_{s}^+}|\nabla^3\bfh|^2\dx\lesssim \int_{B_{r}^+}\frac{|\nabla^2\bfh|^2}{(r-s)^2}\dx.
%\end{
Combining this with \eqref{eq:0303d} and \eqref{eq:0303e} we obtain
\begin{align*}
\int_{B_{s}^+}|\tilde\nabla^2\vartheta|^2\dx+\int_{B_{s}^+}|\nabla^3\bfh|^2\lesssim \int_{B_{r}^+}\frac{|\tilde\nabla^2\bfh|^2}{(r-s)^2}\dx+\int_{B_{r}^+}\frac{|\tilde\nabla\vartheta-(\tilde\nabla\vartheta)_{B_{r}^+}|^2}{(r-s)^2}\dx.
\end{align*}
and, since $\nabla^2\vartheta=\Delta \nabla \bfh$,
\begin{align*}%\label{eq:cacc3}
\int_{B_{s}^+}|\nabla^2\vartheta|^2\dx+\int_{B_{s}^+}|\nabla^3\bfh|^2\lesssim \int_{B_{r}^+}\frac{|\tilde\nabla^2\bfh|^2}{(r-s)^2}\dx+\int_{B_{r}^+}\frac{|\tilde\nabla\vartheta-(\tilde\nabla\vartheta)_{B_r^+})|^2}{(r-s)^2}\dx.
\end{align*}
Following the this idea we can prove similarly
\begin{align}\label{eq:cacc3}
\int_{B_{s}^+}|\nabla^k\vartheta|^2\dx+\int_{B_{s}^+}|\nabla^{k+1}\bfh|^2\lesssim \int_{B_{r}^+}\frac{|\tilde\nabla^k\bfh|^2}{(r-s)^2}\dx+\int_{B_{r}^+}\frac{|\tilde\nabla^{k-1}(\vartheta-(\vartheta)_{B_r^+})|^2}{(r-s)^2}\dx.
\end{align}
for $k\in\N$.
Now we have all at hand to prove the homogeneous decay estimate. It suffices to consider the case $r=1$ since the general case follows by scaling. By Poincar\'e's inequality and Sobolev's embedding for $\ell>\frac{n}{2}$ we obtain for $\theta\leq 1/2$
\begin{align*}
\dashint_{B^+_{\theta}}|\vartheta-(\vartheta)_{B^+_\theta}|^2\dx&+\dashint_{B^+_{\theta}}|\nabla\bfh-(\partial_n\bfh\otimes e_n)_{B_r^+}|^2\dx\lesssim\theta^2 \dashint_{B^+_{\theta}}|\nabla^2\bfh|^2\dx+\theta^2 \dashint_{B^+_{\theta}}|\nabla\vartheta|^2\dx\\&\leq \theta^2\sup_{B^+_{\theta}}|\nabla^2\bfh|^2+\theta^2\sup_{B^+_{\theta}}|\nabla\vartheta|^2\lesssim \theta^{2-n}\big(\|\nabla^2\bfh\|^2_{W^{\ell,2}(B_\theta^+)}+\|\nabla\vartheta\|^2_{W^{\ell,2}(B_\theta^+)}\big)\\
%&\lesssim\theta^2 \dashint_{B^+_{\theta}}\Big(|\nabla^4\bfh|^2+|\nabla^3\bfh|^2+|\nabla^2\bfh|^2\Big)\dx\lesssim \theta^2\dashint_{B^+_{3/2\theta}}\Big(|\nabla^3\bfh|^2+|\nabla^2\bfh|^2\Big)\dx\\
&
\lesssim \theta^2\dashint_{B^+_{2\theta}}|\nabla\bfh-(\partial_n\bfh\otimes e_n)_{B_{2\theta}^+}|^2\dx+\theta^2\dashint_{B^+_{2\theta}}|\vartheta-(\vartheta)_{B_{2\theta}^+}|^2\dx.
\end{align*}
Note that we applied \eqref{eq:cacc3} $\ell$-times.
By scaling we obtain
\begin{align*}
\dashint_{B^+_{\theta r}}|\vartheta-(\vartheta)_{B^+_{\theta r}}|^2\dx&+\dashint_{B^+_{\theta r}}|\nabla\bfh-(\partial_n\bfh\otimes e_n)_{B_{\theta r}^+}|^2\dx\\&\lesssim \theta^2\dashint_{B^+_{r}}|\nabla\bfh-(\partial_n\bfh\otimes e_n)_{B_r^+}|^2\dx+\theta^2\dashint_{B^+_{r}}|\vartheta-(\vartheta)_{B^+_r}|^2\dx
\end{align*}
for $\theta\leq 1/2$, the case $\theta\in(\frac{1}{2},1)$ being trivial.
Combining this with \eqref{eq:comparison}
shows
\begin{align*}
\dashint_{B^+_{\theta r}}&|\nabla\bfu-(\partial_n\bfu\otimes e_n)_{B^+_{\theta r}}|^2\dx+\dashint_{B^+_{\theta r}}|\pi-(\pi)_{B^+_{\theta r}}|^2\dx\\&
\lesssim\dashint_{B^+_{\theta r}}|\nabla\bfh-(\partial_n\bfh\otimes e_n)_{B^+_{\theta r}}|^2\dx+\dashint_{B^+_{\theta r}}|\vartheta-(\vartheta)_{B_{\theta r}}|^2\dx+\dashint_{B^+_{\theta r}}|\nabla\bfu-\bfh|^2\dx+\dashint_{B^+_{\theta r}}|\pi-\vartheta|^2\dx\\
&
\lesssim\theta^2\dashint_{B^+_{ r}}|\nabla\bfh-(\partial_n\bfh\otimes e_n)_{B^+_{r}}|^2\dx+\theta^2\dashint_{B^+_{r}}|\vartheta-(\vartheta)_{B_{r}}|^2\dx+\theta^{-n}\dashint_{B^+_{r}}|\nabla\bfu-\bfh|^2\dx+\theta^{-n}\dashint_{B^+_{r}}|\pi-\vartheta|^2\dx\\
&\lesssim \theta^2 \dashint_{B^+_{r}}|\nabla\bfu-(\partial_n\bfu\otimes e_n)_{B^+_{ r}}|^2\dx+\theta^2\dashint_{B^+_{r}}|\pi-(\pi)_{B_{r}}|^2\dx+c_\theta \dashint_{B^+_{r}}|\bfF-(\bfF)_{B^+_{ r}}|^2\dx.
\end{align*}
Finally, we aim at removing the pressure oscillations from the right-hand side. Arguing similarly to \eqref{eq:comparison2}
we have
\begin{align}\label{eq:comparison2B}
\begin{aligned}
\dashint_{B^+_{r}}|\pi-(\pi)_{B_r^+}|^2\dx&\lesssim r^{-n}\|\nabla\pi\|^2_{W^{-1,2}(B_r^+)}= r^{-n}\|\Div(\nabla\bfu+\bfF)\|_{W^{-1,2}(B_r^+)}^2\\
&\lesssim \dashint_{B^+_{r}}|\nabla\bfu-(\nabla\bfu)_{B_r^+}|^2\dx+ \dashint_{B^+_{r}}|\bfF-(\bfF)_{B^+_{ r}}|^2\dx\\
&\leq\dashint_{B^+_{r}}|\nabla\bfu-(\partial_n\bfu\otimes e_n)_{B^+_{ r}}|^2\dx+ \dashint_{B^+_{r}}|\bfF-(\bfF)_{B^+_{ r}}|^2\dx
\end{aligned}
\end{align}
using minimality of the mean-value in the last step. 
\end{proof}
Similarly to Theorem \ref{thm:decay} we obtain the following result for the problem in non-divergence form. Here for $\bff\in L^2(\mathbb H)$ and $h\in\mathcal D^{1,2}(\mathbb H)$ there is a solution $(\bfu,\pi)\in \mathcal D^{2,2}(\mathbb H)\times \mathcal D^ {1,2}(\mathbb H)$, cf. \cite[Chapter IV, Theorem IV 3.2]{Ga}. It is unique up to a linear function for the velocity field and a possible additive constant for the pressure.

\begin{theorem}\label{thm:decayf}
Suppose that we have $\bff\in L^2(\mathbb H)$ and that
 $(\bfu,\pi)\in \mathcal D^{2,2}(\mathbb H)\times \mathcal D^ {1,2}(\mathbb H)$ is a solution to \eqref{eq:Stokeshalf} with $h=0$. Then we have
\begin{align}\label{est:halfu}
\begin{aligned}
\dashint_{B^+_{\theta r}(x_0)}&|\nabla^2\bfu-(\partial_n\nabla\bfu\otimes e_n)_{B^+_{\theta r}(x_0)}|^2\dx+\dashint_{B^+_{\theta r}(x_0)}|\nabla\pi-(\nabla\pi)_{B^+_{\theta r}(x_0)}|^2\dx\\&\lesssim \theta^2 \dashint_{B^+_{r}(x_0)}|\nabla^2\bfu-(\partial_n\nabla\bfu\otimes e_n)_{B^+_{ r}(x_0)}|^2\dx
%+\theta^2 \dashint_{B^+_{r}(x_0)}|\nabla\pi-(\nabla\pi)_{B^+_{r}(x_0)}|^2\dx\\&
+c_\theta \dashint_{B^+_{r}(x_0)}|\bff-(\bff)_{B^+_{ r}(x_0)}|^2\dx,
\end{aligned}\end{align}
for all $x_0\in \partial\mathbb H$, all $r>0$ and all $\theta\in(0,1)$.
\end{theorem}
\begin{proof}
As in the proof of Theorem \ref{thm:decay} the proof is based on a decay and a comparison estimate. The starting point  
is now \eqref{eq:0303d} for $k=2$ rather than $k=1$.
We obtain
\begin{align}\label{eq:2603}
\begin{aligned}
\dashint_{B^+_{\theta r}}|\nabla\vartheta-(\nabla\vartheta)_{B^+_{\theta r}}|^2\dx&+\dashint_{B^+_{\theta r}}|\nabla^2\bfh-(\partial_n\nabla\bfh\otimes e_3)_{B_{\theta r}^+}|^2\dx\\&\lesssim \theta^2\dashint_{B^+_{\theta}}|\nabla^2\bfh-(\partial_n\nabla\bfh\otimes e_n)_{B_r^+}|^2\dx+\theta^2\dashint_{B^+_{r}}|\nabla\vartheta-(\nabla\vartheta)_{B^+_r}|^2\dx.
\end{aligned}
\end{align}
 Note that exactly the same proof applies if we consider instead of \eqref{eq:Stokeshalfhomo} the problem
\begin{align}\label{eq:Stokeshalfhomo'}
\Delta \bfh-\nabla\vartheta=(\bff)_{B_r^+},\quad\Div\bfu=0,\quad\bfh|_{B_r^+\cap \partial{\mathbb H}}=0,\quad\bfh|_{\partial B_r^+\cap {\mathbb H}}=\bfu,
\end{align}
in $B_r^+$. In this case the difference $\bfv=\bfu-\bfh$, $\mathfrak q=\pi-\vartheta$ solves
\begin{align}\label{eq:Stokeshalfhomo''}
\Delta \bfv-\nabla\mathfrak q=\bff-(\bff)_{B_r^+},\quad\Div\bfv=0,\quad\bfh|_{B_r^+\cap \partial{\mathbb H}}=0,\quad\bfh|_{\partial B_r^+\cap {\mathbb H}}=0,
\end{align}
in $B_r^+$ and thus satsifies the estimate
\begin{align}\label{eq:2603b}
\dashint_{B_r^+}|\nabla^2\bfv|^2\dx+\dashint_{B_r^+}|\nabla\mathfrak q|^2\dx\lesssim 
\dashint_{B_r^+}|\bff-(\bff)_{B_r^+}|^2\dx
\end{align}
by the classical theory (although $B_r^+$ does not have a smooth boundary, it is convex, which is sufficient).
By \eqref{eq:2603} and \eqref{eq:2603b} the proof can be completed analogously to that of Theorem \ref{thm:decay}.
\end{proof}

\subsection{Campanto estimates}
The desired Campanato-estimates for \eqref{eq:Stokeshalf} follow now directly as corollaries of Theorems \ref{thm:decay} and \ref{thm:decayf}.
\begin{corollary}\label{cor:camH}
Let $\omega$ be a parameter function satisfying \eqref{eq:omega condition}.
Suppose that we have $\bfF,h\in L^2(\mathbb H) \cap\mathcal L^\omega(\mathbb H)$
% with $h\in L^2_\perp(\mathbb H)$ 
and that
 $(\bfu,\pi)\in \mathcal D^{1,2}_{0,\Div}(\mathbb H)\times L^2(\mathbb H)$ is the solution to \eqref{eq:Stokeshalf}. Then we have 
\begin{align}
\|\nabla\bfu\|_{\mathcal L^{\omega}(\mathbb H)}+\|\pi\|_{\mathcal L^{\omega}(\mathbb H)}\lesssim \|\bfF\|_{\mathcal L^{\omega}(\mathbb H)}+\|h\|_{L^\omega(\mathbb H)}.
\end{align}
\end{corollary}
\begin{remark}
Note that we even prove a stronger statement. In fact, we have a pointwise estimate in the spirit of \cite{BCDKS}: It holds
\begin{align*}
\mathcal M^{\sharp,2}_{\omega,\mathbb H}(\nabla\bfu)(x_0)\lesssim \mathcal M^{\sharp,2}_{\omega,\mathbb H}(\nabla\bfu)(x_0)
\end{align*}
for all $x_0\in\partial\mathbb H$ (an estimate for interior points can be obtained analogously).
\end{remark}
\begin{proof}[Proof of Corollary \ref{cor:camH}]
Let us first consider the case $h=0$. The general case follows by replacing $\bfu$ by $\bfv:=\bfu-\mathrm{Bog}_{\mathbb H}(h)$, where $\mathrm{Bog}_{\mathbb H}$ is the Bogovskii-operator on $\mathbb H$. As shown in \cite{CM}
it is a pseudodifferential operator belonging to the H\"ormander class $S^{-1}_{1,0}$ and thus $\nabla\mathrm{Bog}_{\mathbb H}$ is continuous on Campanto spaces as proved in \cite{Pe}.

Dividing \eqref{est:halfu} by $\omega(\theta r)$ and using
\eqref{eq:omega condition} shows
\begin{align*}
\frac{1}{\omega(\theta r)}&\dashint_{B^+_{\theta r}(x_0)}|\nabla\bfu-(\partial_n\bfu\otimes e_n)_{B^+_{\theta r}(x_0)}|^2\dx+\frac{1}{\omega(\theta r)}\dashint_{B^+_{\theta r}(x_0)}|\pi-(\pi)_{B^+_{\theta r}(x_0)}|^2\dx\\&\lesssim \theta^{2-\beta_0} \frac{1}{\omega(r)}\dashint_{B^+_{r}(x_0)}|\nabla\bfu-(\partial_n\bfu\otimes e_n)_{B^+_{ r}(x_0)}|^2\dx%+ \frac{1}{\omega(r)}\dashint_{B^+_{r}(x_0)}|\pi-(\pi)_{B^+_{ r}(x_0)}|^2\dx\bigg)\\&
+c_\theta\frac{1}{\omega(r)} \dashint_{B^+_{r}(x_0)}|\bfF-(\bfF)_{B^+_{ r}(x_0)}|^2\dx
\end{align*}
and thus
\begin{align*}
\sup_r\frac{1}{\omega(r)}&\dashint_{B^+_{r}(x_0)}|\nabla\bfu-(\partial_n\bfu\otimes e_n)_{B^+_{ r}(x_0)}|^2\dx+\sup_r\frac{1}{\omega(r)}\dashint_{B^+_{r}(x_0)}|\pi-(\pi)_{B^+_{ r}(x_0)}|^2\dx\\&\lesssim \theta^{2-\beta_0} \sup_r\frac{1}{\omega(r)}\dashint_{B^+_{r}(x_0)}|\nabla\bfu-(\partial_n\bfu\otimes e_n)_{B^+_{ r}(x_0)}|^2\dx
%+ \sup_r\frac{1}{\omega(r)}\dashint_{B^+_{r}(x_0)}|\pi-(\pi)_{B^+_{ r}(x_0)}|^2\dx\bigg)\\&
+c_\theta\sup_r\frac{1}{\omega(r)} \dashint_{B^+_{r}(x_0)}|\bfF-(\bfF)_{B^+_{ r}(x_0)}|^2\dx
\end{align*}
For $\theta$ small enough the first term on the right-hand side can be absorbed and we obtain
\begin{align}\label{eq:2403}
(\mathcal M^{\sharp,2}_{\omega,\mathbb H}\nabla\bfu)(x_0)+(\mathcal M^{\sharp,2}_{\omega,\mathbb H}\pi)(x_0)
\lesssim (\mathcal M^{\sharp,2}_{\omega,\mathbb H}\bfF)(x_0)\leq  \|\bfF\|_{\mathcal L^{\omega}(\mathbb H)}.
\end{align}
Note that we replaced $(\partial_n\bfu\otimes e_n)_{B^+_{ r}(x_0)}$ by $(\nabla\bfu)_{B^+_{ r}(x_0)}$ on the left-hand side using minimality of the mean value.
In order to obtain a corresponding estimate for interior points we have to distinguish between two cases. First let
$x_0\in\mathbb H$ and $r>0$ be such that $B_r(x_0)\subset\mathbb H$. Then a version of \eqref{est:halfu} with full balls is well known (see, e.g., \cite{GM} or \cite[Lemma 3.5]{FuS}), i.e., we have
\begin{align*}
\dashint_{B_{\theta r}(x_0)}|\nabla\bfu-(\nabla\bfu)_{B_{\theta r}(x_0)}|^2\dx&\lesssim \theta^2 \dashint_{B_{r}(x_0)}|\nabla\bfu-(\nabla\bfu)_{B_{ r}(x_0)}|^2\dx\\&+c_\theta \dashint_{B_{r}(x_0)}|\bfF-(\bfF)_{B_{ r}(x_0)}|^2\dx\end{align*}
together with an according estimate for the pressure,\footnote{Arguing as in \eqref{eq:comparison2} and \eqref{eq:comparison2B} one can control the oscillations of the pressure by that of the velocity gradient using the equation together with the negative norm theorem.} that is
\begin{align*}
\dashint_{B_{\theta r}(x_0)}|\nabla\bfu-(\nabla\bfu)_{B_{\theta r}(x_0)}|^2\dx&+\dashint_{B_{\theta r}(x_0)}|\pi-(\pi)_{B_{\theta r}(x_0)}|^2\dx\\&\lesssim \theta^2 (\mathcal M^{\sharp,2}_{\omega,\mathbb H}\nabla\bfu)(x_0)+
%(\mathcal M^{\sharp,2}_{\omega,\mathbb H}\pi)(x_0)\big)
+c_\theta \|\bfF\|_{\mathcal L^{\omega}(\mathbb H)}
\end{align*}
If $B_r(x_0)\cap\partial\mathbb H\neq\emptyset$ there is a point $x_0'\in\partial\mathbb H$ such that $|x_0-x_0'|\leq r$ and
$B_r(x_0)\subset 2B_r^+(x_0')\subset 6B_r(x_0)$. Hence we have
\begin{align*}
\dashint_{B_{\theta r}^+(x_0)}|\nabla\bfu&-(\nabla\bfu)_{B_{\theta r}(x_0)}|^2\dx+\dashint_{B_{\theta r}^+(x_0)}|\pi-(\pi)_{B_{\theta r}(x_0)}|^2\dx\\&\lesssim \dashint_{B^+_{2\theta r}(x_0')}|\nabla\bfu-(\nabla\bfu)_{B^+_{2\theta r}(x_0')}|^2\dx+\dashint_{B^+_{2\theta r}(x_0')}|\pi-(\pi)_{B^+_{2\theta r}(x_0')}|^2\dx\lesssim \|\bfF\|_{\mathcal L^{\omega}(\mathbb H)}
\end{align*}
using \eqref{eq:2403}.
Combining both cases (with an appropriate choice of $\theta$), we have shown
\begin{align}\label{eq:2403b}
(\mathcal M^{\sharp,2}_{\omega,\mathbb H}\nabla\bfu)(x_0)+(\mathcal M^{\sharp,2}_{\omega,\mathbb H}\pi)(x_0)
\lesssim   \|\bfF\|_{\mathcal L^{\omega}(\mathbb H)}
\end{align}
for interior points $x_0\in\mathbb H$. Plugging  \eqref{eq:2403} and \eqref{eq:2403b} together yields the claim taking the supremum with respect to $x_0$.
\end{proof}
We obtain an analogous result for the problem in non-divergence form as a corollary of Theorem \ref{thm:decayf} which can be proved along the same lines.
\begin{corollary}\label{cor:camH2}
Let $\omega$ be a parameter function satisfying \eqref{eq:omega condition}.
Suppose that we have $\bff\in L^2(\mathbb H)\cap \mathcal L^\omega(\mathbb H)$, $h\in \mathcal D^{1,2}(\mathbb H)\cap W^1\mathcal L^\omega(\mathbb H)$ and that
 $(\bfu,\pi)\in \mathcal D^{2,2}(\mathbb H)\times \mathcal D^{1,2}(\mathbb H)$ is a solution to \eqref{eq:Stokeshalf}. Then we have 
\begin{align}
\|\nabla^2\bfu\|_{\mathcal L^{\omega}(\mathbb H)}+\|\nabla\pi\|_{\mathcal L^{\omega}(\mathbb H)}\lesssim \|\bff\|_{\mathcal L^{\omega}(\mathbb H)}+ \|h\|_{W^1\mathcal L^{\omega}(\mathbb H)}.
\end{align}
\end{corollary}

\section{The Stokes equations in rough domains}
This section is devoted to the study of the Stokes equations
in a domain $\mathcal O\subset\R^n$ with minimal regularity in Campanato spaces. The heart of the paper is contained in Section \ref{sec:stokessteady} where we study the problem with right-hand side in divergence-form.

%Section
%\ref{sec:stokesunsteady} is dedicated to the unsteady Stokes system in a fixed domain. Finally, we consider Stokes equations in moving domains in Section
%\ref{sec:stokesunsteadymoving}.
% This is the cornerstone for the higher order estimates for the fluid-structure problem in Section \ref{sec:higher}.

\subsection{The problem in divergence form}\label{sec:stokessteady}
In this section we consider the steady Stokes system
\begin{align}\label{eq:Stokes}
\Delta \bfu-\nabla\pi=-\Div\bfF,\quad\Div\bfu=0,\quad\bfu|_{\partial{\Omega}}=0,
\end{align}
in a domain ${\Omega}\subset\R^n$. The result given in the following theorem is a maximal regularity estimate for the solution in terms of the right-hand side under minimal assumption on the regularity of $\partial\Omega$. The sharpness of the assumptions is demonstrated in \cite[Theorem 2.2]{BCDS} for the Laplace equation.  
\begin{theorem}\label{thm:stokessteady}
Let $\omega$ be a parameter function satisfying \eqref{eq:omega condition} and $\bfF\in\mathcal L^\omega(\Omega)$.
 Suppose that ${\Omega}$ is a bounded $C^1$-domain of class $W^1\mathcal L^\sigma(\delta)$ for some sufficiently small $\delta$, where
\begin{align*}
\sigma(r):=\omega(r)\bigg(\int_{r}^1\frac{\omega(\rho)}{\rho}\,\dd\rho\bigg)^{-1}.
\end{align*}
Then there is a unique solution to \eqref{eq:Stokes} and we have
\begin{align}\label{eq:main}
\|\nabla\bfu\|_{\mathcal L^\omega(\Omega)}+\|\pi\|_{\mathcal L^\omega(\Omega)}\lesssim\|\bfF\|_{\mathcal L^\omega(\Omega)}.
\end{align}
The constant in \eqref{eq:main} depends on the 
$\mathcal L^{\sigma}$-norms of the local charts in the parametrisation of  $\partial\Omega$.
\end{theorem}
\begin{proof}
Since $\mathcal L^{\omega}(\Omega)\hookrightarrow L^2(\Omega)$ a unique weak solution $(\bfu,\pi)\in W^{1,2}_{0,\Div}({\Omega})\times L^2_{\perp}({\Omega})$ to \eqref{eq:Stokes} clearly exists. Let us initially suppose that $\bfu$ and $\pi$ are sufficiently smooth such that the following calculations are justified. We will remove this assumption at the end of the proof.
By assumption there is $\ell\in\mathbb N$ and functions $\varphi_1,\dots,\varphi_\ell\in C^1$ satisfying \ref{A1}--\ref{A3} describing a part $\mathcal U^j$ of the boundary $\partial\Omega$ (such that $\partial\Omega\subset\cup_{j=1}^\ell \mathcal U^j$). In accordance with \eqref{PSI} we define functions $\bfPhi_1,\dots,\bfPhi_\ell\in W^1\mathcal L^\sigma$ along with their inverses $\bfPsi_1,\dots,\bfPsi_\ell$.
% we denote the inverse functions of the $\bfPhi_1,\dots,\bfPhi_\ell$. 
We clearly find an open set $\mathcal U^0\Subset{\Omega}$ such that ${\Omega}\subset \cup_{j=0}^\ell \mathcal U^j$. Finally, we consider a decomposition of unity $(\xi_j)_{j=0}^\ell$ with respect to the covering
$\mathcal U^0,\dots,\mathcal U^\ell$ of ${\Omega}$. 
 
 Let us fix $j\in\{1,\dots,\ell\}$ and assume, without loss of generality, that the reference point $y_j=0$ and that the outer normal at~$0$ is pointing in the negative $x_n$-direction (this saves us some notation regarding the translation and rotation of the coordinate system).
We multiply $\bfu$ by $\xi_j$ and obtain for $\bfu_j:=\xi_j\bfu$, $\Pi_j:=\xi_j\pi$ and $\bfF_j:=\xi_j\bfF$ the equation
\begin{align}\label{eq:Stokes2}
\Delta \bfu_j-\nabla\Pi_j=-\Div\bfF_j+[\Delta,\xi_j]\bfu-[\nabla,\xi_j]\Pi+[\Div,\xi_j]\bfF,\quad\Div\bfu_j=\nabla\xi_j\cdot\bfu,\quad\bfu_j|_{\partial{\Omega}}=0,
\end{align}
with the commutators $[\Delta,\xi_j]=\Delta\xi_j+2\nabla\xi_j\cdot\nabla$, $[\nabla,\xi_j]=\nabla\xi_j$ and $[\Div,\xi_j]=\nabla\xi_j\cdot$.
Finally, we set $\bfv_j:=\bfu_j\circ\bfPsi_j$, $\theta_j:=\Pi_j\circ\bfPsi_j$, $h_j=(\nabla\xi_j\cdot\bfu)\circ\bfPsi_j$ as well as\footnote{We denote by $\Delta^{-1}_{\mathbb H}$ the solution operator to the Laplace equation in $\mathbb H$ with homogenous Dirichlet boundary conditions on $\partial\mathbb H$.}
\begin{align*}
\bfG_j:=\nabla\Delta^{-1}_{\mathbb H}([\Delta,\xi_j]\bfu-[\nabla,\xi_j]\Pi+[\Div,\xi_j]\bfF)\circ\bfPsi_j,\quad\bfH_j:=\mathbf{B}_j\bfF_j\circ\bfPsi_j,
\end{align*}
and obtain the equations
\begin{align}\label{eq:Stokes3}
\begin{aligned}
&\Div\big(\bfA_j\nabla\bfv_j)-\Div(\mathbf{B}_j\theta_j)=-\Div(\bfH_j+\bfG_j),
\quad\mathbf{B}_j^\top:\nabla\bfv_j=h_j,\quad\bfv_j|_{\partial\mathbb H}=0,
\end{aligned}
\end{align}
in $\mathbb H$, where $\bfA_j:=\nabla\bfPhi_j^\top\circ\bfPsi_j\nabla\bfPhi_j\circ\bfPsi_j$ and $\mathbf{B}_j:=\nabla\bfPhi_j\circ\bfPsi_j$
 (note that we have $\Div\mathbf{B}_j=0$ due to the Piola identity).
%$\bfB_j^\top:\nabla\bfv_j=\Div(\bfB_j^\top\bfv_j)=0$, where the last equality is a consequence of the Piola identity
 This can be rewritten as
\begin{align}\label{eq:Stokes3}
\begin{aligned}
\Delta\bfv_j-\nabla\theta_j&=\Div\big((\mathbb I_{n\times n}-\bfA_j)\nabla\bfv_j)+\Div((\mathbf{B}_j-\mathbb I_{n\times n})\theta_j)-\Div(\bfH_j+\bfG_j),\\
&\Div\bfv_j=(\mathbb I_{n\times n}-\mathbf{B}_j)^\top:\nabla\bfv_j+h_j,\quad\bfv_j|_{\partial\mathbb H}=0.
\end{aligned}
\end{align}
Setting 
\begin{align*}
\mathcal S (\bfv,\theta)&=\mathcal S _1(\bfv)+\mathcal S_2 (\theta),\\
\mathcal S_1 (\bfv)&=(\mathbb I_{n\times n}-\bfA_j)\nabla\bfv,\\
%\mathcal S_2 (\bfv)&=-\Div(\bfA_j\nabla\mathrm{Bog}_{\mathbb H}\big((\mathbb I_{n\times n}-\bfB_j)^\top:\nabla\tilde\bfv_j\big)),\\
\mathcal S_2 (\theta)&=(\mathbf{B}_j-\mathbb I_{n\times n})\theta,\\
\mathfrak s(\bfv)&=(\mathbb I_{n\times n}-\mathbf{B}_j)^\top:\nabla\bfv,
\end{align*} we can finally write
\eqref{eq:Stokes3} as
\begin{align}\label{eq:Stokes3'}
\begin{aligned}
\Delta\bfv_j-\nabla\theta_j=\Div\mathcal S (\bfv_j,\theta_j)-\Div(\bfH_j+\bfG_j),
\quad\Div\bfv_j=\mathfrak s(\bfv_j)+h_j,\quad\bfv_j|_{\partial\mathbb H}=0,
\end{aligned}
\end{align}
in $\mathbb H.$
We can now apply the estimates for the half space from Corollary \ref{cor:camH} obtaining
\begin{align}\label{eq:0201c}
\|\nabla\bfv_j\|_{\mathcal L^\omega(\mathbb H)}+\|\theta_j\|_{\mathcal L^\omega(\mathbb H)}\lesssim \|\mathcal S (\bfv_j,\theta_j)-\bfH_j-\bfG_j\|_{\mathcal L^\omega(\mathbb H)}+\|\mathfrak s(\bfv_j)+h_j\|_{\mathcal L^\omega(\mathbb H)}.
\end{align}
Our remaining task consists in estimating the right-hand side.
%(for $s=1$ and $\varrho=p$). Interpolation gives homogeneous Besov spaces.
%This gives semi-norm version of \eqref{eq:main}. Since zero boundary conditions we get norm version.
%\begin{align*}
%(\bfv_j,\theta_j)=\mathscr A_{\mathbb H}^{-1}\mathcal S(\bfv_j,\theta_j)+\mathscr A_{\mathbb H}^{-1}\bfg_j
%\end{align*}
%or
%\begin{align*}
%(\bfv_j,\theta_j)=\big(\mathrm{id}-\mathscr A_{\mathbb H}^{-1}\mathcal S\big)^{-1}\mathscr A^{-1}_{\mathbb H}\bfg_j
%\end{align*}
%where $\mathscr A_{\mathbb H}^{-1}$ gives the solution to the Stokes system on $\mathbb H$ (velocity field and pressure) subject 
%to zero boundary conditions and divergence-free constraint.
% We can further write
%\begin{align*}
%(\bfv_j,\theta_j)=\sum_{k\in\mathbb N}\big(\mathscr A_{\mathbb H}^{-1}\mathcal S\big)^{k}\mathcal A_{\mathbb H}^{-1}\bfg_j.
%\end{align*}
%It is well-known that the operator $\mathscr A_{\mathbb H}^{-1}$ satisfies
%\begin{align}\label{eq:0612}
%\|\mathscr A_{\mathbb H}^{-1}\bfg\|_{\bfB^s_{\varrho,p}(\mathbb H)\times\bfB^{s-1}_{\varrho,p}(\mathbb H)}\lesssim \|\bfg\|_{\bfB^{s-2}_{\varrho,p}(\mathbb H)}.
%\end{align}
In order to do so we employ Campanato multipliers as introduced in Section \ref{subsec:cs}.
By assumption we have $\varphi_j\in \mathcal L^\sigma$ with norm bounded by $\delta$ such that, as a consequence of \eqref{eq:cm},
the $\mathscr M^\omega$-norm of $\nabla\varphi_j$ is bounded by $\delta$ as well (note that the $L^1$-part in \eqref{eq:cm} can be neglected as we have $\nabla\varphi_j(0)=0$).
 Hence we obtain
by \eqref{J}, \eqref{lem:9.4.1} and the definitions of $\bfA_j$ and $\bfPsi_j$ 
%\begin{align*}
%\|\mathcal S(\bfv,\theta)\|_{\bfB^{s-2}_{\varrho,p}(\mathbb H)}&\lesssim \|\Div\big((\mathbb I_{n\times n}-\bfA_j)\nabla\bfv_j)\|_{\bfB^{s-2}_{\varrho,p}(\mathbb H)}+\|\Div((\bfB_j-\mathbb I_{n\times n})\theta)\|_{\bfB^{s-2}_{\varrho,p}(\mathbb H)}
%\end{align*}
\begin{align*}
\|\mathcal S_1(\bfv_j)\|_{\mathcal L^\omega(\mathbb H)}&\lesssim \|\mathbb I_{n\times n}-\bfA_j\|_{\mathscr M^\omega(\bfPhi_j(\mathcal U^j\cap \Omega))}\|\nabla\bfv_j\|_{\mathcal L^\omega(\mathbb H)}\\
&\lesssim\|\mathbb I_{n\times n}-\nabla\bfPhi_j^\top\circ\bfPsi_j\|_{\mathscr M^\omega(\bfPhi_j(\mathcal U^j\cap \Omega))}\|\nabla\bfv_j\|_{\mathcal L^\omega(\mathbb H)}\\&+\|\nabla\bfPhi_j^\top\circ\bfPsi_j(\mathbb I_{n\times n}-\nabla\bfPhi_j\circ\bfPsi_j)\|_{\mathscr M^\omega((\bfPhi_j(\mathcal U^j\cap \Omega)))}\|\nabla\bfv_j\|_{\mathcal L^\omega(\mathbb H)}\\
&\lesssim\big(1+\|\nabla\bfPhi_j^\top\circ\bfPsi_j\|_{\mathscr M^\omega(\bfPhi_j(\mathcal U^j))}\big)\|\mathbb I_{n\times n}-\nabla\bfPhi_j^\top\circ\bfPsi_j\|_{\mathscr M^\omega(\bfPhi_j(\mathcal U^j))}\|\nabla\bfv_j\|_{\mathcal L^\omega(\mathbb H)},
\end{align*}
where, by \eqref{lem:9.4.1},
\begin{align*}
\|\nabla\bfPhi_j^\top\circ\bfPsi_j\|_{\mathscr M^\omega(\bfPhi_j(\mathcal U^j\cap \Omega))}
%=\|\mathrm{cof}(\nabla\bfPsi_j)\|_{\mathscr M^\omega(\mathbb H)}
\lesssim \|\nabla\bfPhi_j\|_{\mathscr M^\omega(\mathbb H)}\lesssim 1+\|\nabla\varphi_j\|_{\mathscr M^\omega(\partial\mathbb H)}\lesssim 1.
\end{align*}
%by \eqref{lem:9.4.1}.
Furthermore, it holds
\begin{align*}
\|(\mathbb I_{n\times n}-\nabla\bfPhi_j\circ\bfPsi_j)\|_{\mathscr M^\omega(\bfPhi_j(\mathcal U^j\cap \Omega))}\lesssim \|\nabla\varphi_j\circ\bfPsi_j\|_{\mathscr M^\omega(\partial\mathbb H))}
%=\|\mathbb I_{n\times n}-\mathrm{cof}(\nabla\bfPsi_j)\|_{\mathscr M^\omega(\mathbb H)}\\
&\lesssim \|\nabla\varphi_j\|_{\mathscr M^\omega(\partial\mathbb H))}\lesssim \delta.
\end{align*}
So we finally have
\begin{align*}
\|\mathcal S_1(\bfv_j)\|_{\mathcal L^\omega(\mathbb H)}&\lesssim \delta\|\nabla\bfv_j\|_{\mathcal L^\omega(\mathbb H)}
\end{align*}
and, similarly, 
%by continuity of $\mathrm{Bog}_{\mathbb H}$
\begin{align*}
\|\mathcal S_2(\theta_j)\|_{\mathcal L^\omega(\mathbb H)}&\lesssim \|\mathbf{B}_j-\mathbb I_{n\times n}\|_{\mathscr M^\omega(\bfPhi_j(\mathcal U^j\cap \Omega))}\|\theta_j\|_{\mathcal L^\omega(\mathbb H)}
\lesssim \delta\|\theta_j\|_{ \mathcal L^\omega(\mathbb H)},\\
\|\bfH_j\|_{\mathcal L^\omega(\mathbb H)}&\lesssim \|\mathbf{B}_j\|_{\mathscr M^\omega(\bfPhi_j(\mathcal U^j\cap \Omega))}\|\bfF_j\circ\bfPsi_j\|_{\mathcal L^\omega(\mathbb H)}
\lesssim \|\bfF_j\|_{ \mathcal L^\omega(\mathbb H)}\lesssim \|\bfF\|_{ \mathcal L^\omega(\mathbb H)},\
\end{align*}
as well as
\begin{align*}
\|\mathfrak s(\bfv_j)\|_{\mathcal L^\omega(\mathbb H)}
&\lesssim \|\mathbf{B}_j-\mathbb I_{n\times n}\|_{\mathscr M^\omega(\bfPhi_j(\mathcal U^j\cap \Omega))}\|\nabla\bfv_j\|_{\mathcal L^\omega(\mathbb H)}\lesssim\delta\|\nabla\bfv_j\|_{\mathcal L^\omega(\mathbb H)},
\end{align*}
We conclude that
 \begin{align}\label{eq:0301}
 \|\mathcal S(\bfv_j,\theta_j)\|_{\mathcal L^\omega(\mathbb H)}+\|\mathfrak s(\bfv_j)\|_{\mathcal L^\omega(\mathbb H)}\leq \delta\big(\|\nabla\bfv_j\|_{\mathcal L^\omega(\mathbb H)}+\|\theta_j\|_{\mathcal L^\omega(\mathbb H)}\big)
 \end{align}
 for some small $\delta>0$. 
In order to control the lower order term related to $\bfG_j$ we consider $\beta_0$ given in \eqref{eq:omega condition} and set
$q:=\frac{n}{1-\beta_0}$ such that the embedding $W^{1,q}\hookrightarrow C^{0,\beta_0}$ holds. Note also that we trivially have $C^{0,\beta_0}\hookrightarrow\mathcal L^{t^{\beta_0}}\hookrightarrow\mathcal L^\omega$.
These facts and the continuity of $\nabla^2\Delta_{\mathbb H}^{-1}$ on $L^{q}(\mathbb H)$ yield
 \begin{align*}
 \|\bfG_j\|_{\mathcal L^\omega(\mathbb H)}&\lesssim \|\bfG_j\|_{C^{0,\beta_0}(\mathbb H)}\lesssim  \|\bfG_j\|_{W^{1,q}(\mathbb H)}\\
&\lesssim\|\nabla\bfu\circ\bfPsi_j\|_{L^{q}(\bfPhi_j(\mathcal U^j\cap \Omega))}+\|\pi\circ\bfPsi_j\|_{L^{q}(\bfPhi_j(\mathcal U^j\cap \Omega))}+ \|\bfF\circ\bfPsi_j\|_{L^{q}(\bfPhi_j(\mathcal U^j\cap \Omega))}\\
&\lesssim\|\nabla\bfu\|_{L^{q}(\Omega)}+\|\pi\|_{L^{q}(\Omega)}+ \|\bfF\|_{L^{q}(\Omega)}
 \end{align*}
%%\cite[Theorem IV. 6.1]{Gi} needs $C^2$!
cf.~ \eqref{lem:9.4.1}. By the $L^q$-theory for the Stokes system in Reifenberg-flat domains (which include $C^1$-domains) from \cite{BH} we finally obtain
 \begin{align*}
 \|\bfG_j\|_{\mathcal L^\omega(\mathbb H)}&\lesssim \|\bfF\|_{L^{q}(\Omega)}\lesssim \|\bfF\|_{ \mathcal L^{\omega}(\Omega)}
 \end{align*}
noticing that we can assume that $(\bfF)_{\Omega}=0$.
% Note that the first inequality is an elementary consequence of the definition of the $\mathcal L^\omega$-semi-norm considering balls with radius $\leq \delta$ and balls with radius $\geq\delta$ separately.
 Similarly, we obtain
 \begin{align*}
 \|h_j\|_{\mathcal L^\omega(\mathbb H)}&\lesssim \|\bfF\|_{\mathcal L^\omega(\Omega)}.
 \end{align*}

 Plugging this and \eqref{eq:0301} into \eqref{eq:0201c} shows for all $j\in\{1,\dots,\ell\}$
 \begin{align}\label{almost0}
\|\nabla\bfv_j\|_{\mathcal L^\omega(\mathbb H)}+\|\theta_j\|_{\mathcal L^\omega(\mathbb H)}\lesssim \delta
 \|\nabla\bfu\|_{\mathcal L^\omega(\mathbb H)}+ \delta \|\pi\|_{\mathcal L^\omega(\mathbb H)}+\|\bfF\|_{\mathcal L^\omega(\Omega)}.
\end{align}
%provided $\delta$ is sufficiently small. 
Clearly, the same estimate without the first two terms on the right-hand side holds for $j=0$ by local regularity theory for the Stokes system. In fact, one can argue as in the proof of Theorem \ref{thm:decay} and combine well-known homogeneous decay estimates (see, e.g., \cite{GM} or \cite[Lemma 3.5]{FuS}) with a comparison estimate and eventually argue similarly to the proof of Corollary \ref{cor:camH} to obtain local Campanato estimates. Such a strategy is used in \cite{DKS2} for Stokes systems related to non-Newtonian fluids. 
%Choosing $s_0\in\R$ such that $W^{1,2}({\Omega})\hookrightarrow W^{s_0,p}({\Omega})$, there is $\alpha\in(0,1)$ such that
% \begin{align*}
% \|\bfu\|_{W^{s-1,p}_x}&\leq \|\bfu\|_{W^{s,p}_x}^{\alpha}\|\bfu\|_{W^{s_0,p}_x}^{1-\alpha}\lesssim \|\bfu\|_{W^{s,p}_x}^{\alpha}\|\bfu\|_{W^{1,2}_x}^{1-\alpha} \\&\lesssim\|\bfu\|_{W^{s,p}_x}^{\alpha}\|\bff\|_{W^{-1,2}_x}^{1-\alpha}\lesssim\|\bfu\|_{W^{s,p}_x}^{\alpha}\|\bff\|_{W^{s-2,p}_x}^{1-\alpha}
% \end{align*}
% by the assumption \small$n\big(\frac{1}{p}-\frac{1}{2}\big)+1\leq  s$\normalsize\, and the standard energy estimate for the Stokes system. 
Hence we obtain, transforming back to $\Omega$, using \eqref{lem:9.4.1} and writing $\bfu=\sum_j \bfu_j$ as well as $\pi=\sum_j \pi_j$  
  \begin{align}\label{almost}
\|\nabla\bfu\|_{\mathcal L^\omega(\Omega)}+\|\pi\|_{\mathcal L^\omega(\Omega)}\lesssim \delta
 \|\nabla\bfu\|_{\mathcal L^\omega(\Omega)}+ \delta \|\pi\|_{\mathcal L^\omega(\Omega)}+\|\bfF\|_{\mathcal L^\omega(\Omega)}.
 \end{align}
%Similarly,
%  \begin{align}\label{almost2}
% \|\pi\|_{W^{s-2,p}_x}&\leq\kappa\|\pi\|_{W^{s-1,p}_x}+c(\kappa)\|\bff\|_{W^{s-2,p}_x}
% \end{align}
% using that $\|\pi\|_{L^2_x}\lesssim \|\bff\|_{W^{-1,2}_x}$ as well.
% Plugging \eqref{almost1} and \eqref{almost2} into \eqref{almost0} summing over $j=0,1,\dots,\ell$ and choosing
If $\delta$ is small enough this proves the claim provided $\bfu$ and $\pi$ are sufficiently smooth.

In the last step we are going to remove this additional smoothness assumption by a standard regularisation procedure as in \cite{Br} (see also \cite[Section 4]{CiMa}). 
Applying a mollifying kernel to the functions $\varphi_1,\dots,\varphi_\ell$ from \ref{A1}--\ref{A3} in the parametrisation of $\partial{\Omega}$ we obtain a smooth boundary. This leads to a sequence of smooth solutions for which the previous computations are fully justified. It is possible to do this in a way that the original domain is included in the regularised domain to which we extend the function $\bfF$ by means of an extension operator. Now it is sufficient to know that the Campanato semi-norm of the boundary charts does not expand when regularising them, which follows easily from the definition of the former.
% The regularisation applied to the $\varphi_j's$ converges on all Besov spaces with $p<\infty$. It does not converge on $W^{1,\infty}(\R^{n-1})$, but the regularisation does not expand the $W^{1,\infty}(\R^{n-1})$-norm, which is sufficient. Following the arguments above we obtain
% \eqref{eq:main} for the regularised problem with a uniform constant. The limit passage is straightforward since \eqref{eq:Stokes} is linear.
\end{proof}

The following Theorem gives an improvement of Theorem \ref{thm:stokessteady} under the additional assumption that the parameter function $\omega$ satisfies the Dini condition: there is no gap between the regularity of the boundary charts (located in $\mathcal L^\sigma$) and that of the forcing (located in $\mathcal L^\omega$) and no smallness condition is needed. It is a counterpart of \cite[Theorem 2.6]{BCDS}, where a corresponding result is proved for the ($p$-)Laplace equation.
\begin{theorem}\label{thm:stokessteady2}
Let $\omega$ be a parameter function satisfying \eqref{eq:omega condition} as well as
\begin{equation}\label{dini'} \int _0 \frac{\omega (r)}r\, {\rm d} r < \infty,
\end{equation}
 and $\bfF\in\mathcal L^\omega(\Omega)$.
 Suppose that ${\Omega}$ is a bounded $C^1$-domain of class $W^1\mathcal L^\omega$.
Then there is a unique solution to \eqref{eq:Stokes} and we have
\begin{align}\label{eq:main'}
\|\nabla\bfu\|_{\mathcal L^\omega(\Omega)}+\|\pi\|_{\mathcal L^\omega(\Omega)}\lesssim\|\bfF\|_{\mathcal L^\omega(\Omega)}.
\end{align}
The constant in \eqref{eq:main'} depends on the 
$\mathcal L^{\omega}$-norms of the local charts in the parametrisation of  $\partial\Omega$.
\end{theorem}
\begin{proof}
The crucial point is that, with Theorem \ref{thm:stokessteady} at hand, we can prove that $\nabla\bfu\in L^\infty(\Omega)$ by following the ideas of \cite{BCDS}. Owing to assumption \eqref{dini'},
%since we are assuming that 
%$$\int_0 \frac{\omega(r)}r\, dr <\infty\,$$
there exists an increasing function $\eta : (0, 1) \to (0, \infty)$ such that $\lim _{r \to 0^+} \eta (r) =0$, and still
\begin{equation}\label{july30}
\int_0 \frac{\omega(r)}{\eta (r)r}\, \mathrm{d}r <\infty\,.
\end{equation}
Next, we define the non-decreasing function $\omega_1 : (0, 1) \to (0,\infty)$ as 
$$\omega_1 (r) = \inf _{s\geq r} \frac{\omega (s)}{\eta (s)} \quad \hbox{for $r> 0$.}$$
Since $\omega_1(r)\leq \omega(r)$ for $r>0$ we have that $\partial \Omega \in W^1 \mathcal L^{\omega_1}$.
%One has that
%\begin{equation}\label{july31}
%c \omega (r) \leq \omega_1 (r) \leq \frac{\omega (r)}{\eta (r)}\quad \hbox{for $r >0$,}
%\end{equation}
%for some positive constant $c$.
%Moreover, the first inequality in \eqref{july31} tells us that $\partial \Omega \in W^1 \mathcal L^{\omega_1(\cdot)}$.
% In particular, the second inequality in \eqref{july31} and condition \eqref{july30} ensure that condition \eqref{dini}
Moreover, it is shown in
\cite[proof of Theorem 2.6]{BCDS} that $\omega_1$ still satisfies \eqref{dini'} and that we have
\begin{equation}\label{july38}
\lim _{r \to 0^+} \frac{\omega (r)}{\omega _1(r)} \int _r^1\frac{\omega_1(s)}{s}\,\mathrm{d}s =0.
\end{equation}
Hence it holds $\partial\Omega\in W^1\mathcal L^{\sigma_1}(\delta)$ for all $\delta>0$, where
\begin{align*}
\sigma_1(r):=\omega_1(r)\bigg(\int_\rho^1\frac{\omega_1(\rho)}{\rho}\,\dd\rho\bigg)^{-1}.
\end{align*}
We apply now Theorem \ref{thm:stokessteady} with $\omega$ replaced by $\omega_1$ and $\sigma_1$ as given above. Owing to \eqref{spanne1} we have $\mathcal L^{\omega_1}\hookrightarrow L^\infty$ and thus
\begin{align}\label{eq.1903}
\begin{aligned}
\|\nabla\bfu\|_{L^\infty(\Omega)}+\|\pi\|_{L^\infty(\Omega)}&\lesssim\|\nabla\bfu\|_{L^{\omega_1}(\Omega)}+\|\pi\|_{L^{\omega_1}(\Omega)}\\&\lesssim \|\bfF\|_{\mathcal L^{\omega_1}(\Omega)}\lesssim \|\bfF\|_{\mathcal L^{\omega}(\Omega)}.
\end{aligned}
\end{align}
With this at hand we enter the proof of Theorem \ref{thm:stokessteady} for the estimates of $\mathcal S(\bfv,\vartheta)$ and $\mathfrak s(\bfv)$.
By \eqref{eq.1903},
estimate \eqref{eq:cm+} yields
\begin{align*}
\|\mathcal S_1(\bfv_j,\theta_j)\|_{\mathcal L^\omega(\mathbb H)}&\lesssim \|\mathbf{A}_j-\mathbb I_{n\times n}\|_{L^\infty(\bfPhi_j(\mathcal U^j\cap \Omega))}\|\nabla\bfv_j\|_{\mathcal L^\omega(\mathbb H)}+\|\mathbf{A}_j-\mathbb I_{n\times n}\|_{\mathcal L^\omega(\bfPhi_j(\mathcal U^j\cap \Omega))}\|\nabla\bfv_j\|_{L^\infty(\mathbb H)}\\
&\lesssim \delta\|\nabla\bfv_j\|_{ \mathcal L^\omega(\mathbb H)}+\|\bfF\|_{\mathcal L^{\omega}(\Omega)}
\end{align*}
using that $\partial\Omega\in W^1\mathcal L^\omega$ and thus has a small local Lipschitz constant by \eqref{spanne1}.
It can be shown with the same idea that
\begin{align*}
\|\mathcal S_2(\theta_j)\|_{\mathcal L^\omega(\mathbb H)}&
\lesssim \delta\|\theta_j\|_{ \mathcal L^\omega(\mathbb H)},\quad
\|\mathfrak s(\bfv_j)\|_{\mathcal L^\omega(\mathbb H))}
\lesssim\delta\|\nabla\bfv_j\|_{\mathcal L^\omega(\mathbb H)}.
\end{align*}
The rest of the proof can now be completed as before.
\end{proof}

\subsection{The problem in non-divergence form}
\label{sec:stokessteadyf}
In this section we consider the steady Stokes system
\begin{align}\label{eq:Stokesf}
\Delta \bfu-\nabla\pi=-\bff,\quad\Div\bfu=0,\quad\bfu|_{\partial{\Omega}}=0,
\end{align}
in a domain ${\Omega}\subset\R^n$. Here we gain two additional derivatives for the velocity field and one for the pressure  compared to the forcing. This requires roughly one derivative more for the boundary charts compared to Theorem  \ref{thm:stokessteady}. Note, however, that there is no gap between the function space $\mathcal L^ \omega$ in the estimate and that for the boundary charts. 
\begin{theorem}\label{thm:stokessteadyf}
Let $\omega$ be a parameter function satisfying \eqref{eq:omega condition} and $\bff\in\mathcal L^\omega(\Omega)$.
 Suppose that ${\Omega}$ is a bounded domain of class $W^2\mathcal L^\omega$.
Then there is a unique solution to \eqref{eq:Stokesf} and we have
\begin{align}\label{eq:mainf}
\|\nabla^2\bfu\|_{\mathcal L^\omega(\Omega)}+\|\nabla\pi\|_{\mathcal L^\omega(\Omega)}\lesssim\|\bff\|_{\mathcal L^\omega(\Omega)}+|(\bff)_\Omega|.
\end{align}
The constant in \eqref{eq:mainf} depends on the 
$\mathcal L^{\omega}$-norms of the local charts in the parametrisation of  $\partial\Omega$.
\end{theorem}
\begin{proof}
Arguing as in the proof of Theorem \ref{thm:stokessteady} we obtain a variant of the system \eqref{eq:Stokes3}, where the term $\Div(\bfF_j+\bfG_j)$ is replaced by $\bff_j+\bfg_j$ with $\bff_j:=(\xi_j\bff)\circ\bfPsi_j$ and $\bfg_j:=([\Delta,\xi_j]\bfu-[\nabla,\xi_j]\Pi+[\Div,\xi_j]\bfF)\circ\bfPsi_j$. Corollary \ref{cor:camH2} yields
\begin{align}\label{eq:0201c'}
\|\nabla^2\bfv_j\|_{\mathcal L^\omega(\mathbb H)}+\|\nabla\theta_j\|_{\mathcal L^\omega(\mathbb H)}\lesssim \|\Div\mathcal S (\bfv_j,\theta_j)+\bff_j+\bfg_j\|_{\mathcal L^\omega(\mathbb H)}+\|\mathfrak s(\bfv_j)+h_j\|_{W^1\mathcal L^\omega(\mathbb H)}.
\end{align}
Clearly, $\|\bff_j\|_{\mathcal L^\omega(\mathbb H)}\leq \|\bff\|_{\mathcal L^\omega(\mathbb H)}+|(\bff)_\Omega|$, while $\bfg_j$ an $h_j$ are lower order terms. Arguing as in the estimate for $\bfG_j$ in the proof of Theorem \ref{thm:stokessteady} we have
\begin{align*}
\|\bfg_j\|_{\mathcal L^\omega(\mathbb H)}&\lesssim \|\nabla\bfv_j\|_{\mathcal L^\omega(\mathbb H)}+\|\theta_j\|_{\mathcal L^\omega(\mathbb H)}\\
&\lesssim\|\nabla^2\bfu\|_{L^{q}(\Omega)}+\|\nabla\pi\|_{L^{q}(\Omega)}+ \|\bff\|_{L^{q}(\Omega)}.
 \end{align*}
Now we use the $L^q$-theory for the steady Stokes system from \cite[Theorem 3.1]{Br}.
 This is necessary since the embedding $W^2\mathcal L^\omega\hookrightarrow C^2$ fails in some cases (such that the results from \cite[Chapter IV]{Ga} do not apply), but we always have
\begin{align*}
W^2\mathcal L^\omega\hookrightarrow W^2\mathrm{BMO}\hookrightarrow \bigcap _{q<\infty}W^{2,q},
\end{align*}
which is sufficient in \cite{Br}. This yields
\begin{align*}
\|\bfg_j\|_{\mathcal L^\omega(\mathbb H)}\lesssim \|\bff\|_{L^q(\Omega)}.
\end{align*}
We can estimate $h_j$ similarly.
 Hence it remains to control the $W^1\mathcal  L^\omega$-norm of $\mathcal S (\bfv_j,\theta_j)$. It holds by \eqref{eq:cm}
\begin{align*}
\|\nabla\mathcal S_1(\bfv_j)\|_{\mathcal L^\omega(\mathbb H)}&\lesssim \|\mathbb I_{n\times n}-\bfA_j\|_{\mathscr M^\omega(\bfPhi_j(\mathcal U^j\cap \Omega))}\|\nabla^2\bfv_j\|_{\mathcal L^\omega(\mathbb H)}+\|\nabla\bfv_j\|_{\mathscr M^\omega(\mathbb H)}\|\nabla\bfA_j\|_{\mathcal L^\omega(\bfPhi_j(\mathcal U^j\cap \Omega))}\\
&\lesssim \|\mathbb I_{n\times n}-\bfA_j\|_{\mathcal L^\sigma(\bfPhi_j(\mathcal U^j\cap \Omega))}\|\nabla^2\bfv_j\|_{\mathcal L^\omega(\mathbb H)}+\|\nabla\bfv_j\|_{\mathcal L^\sigma(\mathbb H)}\|\nabla\bfA_j\|_{\mathcal L^\omega(\bfPhi_j(\mathcal U^j\cap \Omega))}\\
&\lesssim \|\nabla\varphi_j\|_{\mathcal L^\sigma(\partial\mathbb H)}\|\nabla^2\bfv_j\|_{\mathcal L^\omega(\mathbb H)}+\|\nabla\bfv_j\|_{\mathcal L^\sigma(\mathbb H)} \|\nabla^2\varphi_j\|_{\mathcal L^\omega(\partial\mathbb H)}.
\end{align*}
Now we take $q_0=\frac{n}{1-\vartheta_0}$ where $\vartheta_0$ is chosen such that $\sigma$ given by
\begin{align*}
\sigma(r):=\omega(r)\bigg(\int_\rho^1\frac{\omega(\rho)}{\rho}\,\dd\rho\bigg)^{-1}
\end{align*}
 satisfies \eqref{eq:omega condition}. We use the embeddings $$W^1\mathcal L^\omega\hookrightarrow \mathcal L^\sigma(\delta),\quad W^{1,q_0}\hookrightarrow \mathcal L^\sigma,$$ obtaining similarly to the above
\begin{align*}
\|\nabla\mathcal S_1(\bfv_j)\|_{\mathcal L^\omega(\mathbb H)}
&\lesssim \delta\|\nabla^2\bfv_j\|_{\mathcal L^\omega(\mathbb H)}+\|\nabla^2\bfv_j\|_{L^{q_0}(\mathbb H)}\\
&\lesssim \delta\|\nabla^2\bfv_j\|_{\mathcal L^\omega(\mathbb H)}+\|\nabla^2\bfu\|_{ L^{q_0}(\Omega)}\\
&\lesssim \delta\|\nabla^2\bfv_j\|_{\mathcal L^\omega(\mathbb H)}+\|\bff\|_{ L^{q_0}(\Omega)}\\
&\lesssim \delta\|\nabla^2\bfv_j\|_{\mathcal L^\omega(\mathbb H)}+\|\bff\|_{\mathcal L^{\omega}(\Omega)}+|(\bff)_\Omega|,
\end{align*}
where we used again the $L^q$-theory from \cite[Theorem 3.1]{Br}.
We obtain analogously
\begin{align*}
\|\nabla\mathcal S_2(\theta_j)\|_{\mathcal L^\omega(\mathbb H)}
&\lesssim \delta\|\nabla\theta_j\|_{\mathcal L^\omega(\mathbb H)}+\|\bff\|_{\mathcal L^{\omega}(\mathbb H)}+|(\bff)_\Omega|,\\
\|\nabla\mathfrak s_1(\bfv_j)\|_{\mathcal L^\omega(\mathbb H)}
&\lesssim \delta\|\nabla^2\bfv_j\|_{\mathcal L^\omega(\mathbb H)}+\|\bff\|_{\mathcal L^{\omega}(\mathbb H)}+|(\bff)_\Omega|.
\end{align*}
The proof can now be completed similarly to that of Theorem \ref{thm:stokessteady}.
\end{proof}

\smallskip
\par\noindent
{\bf Acknowledgement}. 
The author 
wishes to thank O. Dominguez Bonilla for useful suggestions regarding operators on Campanato spaces, in particular, for pointing out the reference \cite{Pe}.
Moreover, he thanks the anonymous referee for the careful reading of the manuscript and the valuable suggestions.

																													\section*{Compliance with Ethical Standards}\label{conflicts}
\smallskip
\par\noindent
{\bf Conflict of Interest}. The author declares that he has no conflict of interest.

\smallskip
\par\noindent
{\bf Data Availability}. Data sharing is not applicable to this article as no datasets were generated or analysed during the current study.

\end{document}